\documentclass[11pt,reqno]{amsart}

\usepackage{amsmath,amssymb,amsfonts,mathtools}
\usepackage{tikz}
\usetikzlibrary{positioning}
\usepackage{booktabs}
\usepackage{siunitx}
\newcommand{\numeoc}[1]{\num[round-precision=2,round-mode=places, scientific-notation=false, retain-zero-exponent=false]{#1}}
\newcommand{\nummesh}[1]{\num[round-precision=0,round-mode=places, scientific-notation=false, retain-zero-exponent=false]{#1}}

\theoremstyle{plain}
\newtheorem{theorem}{Theorem}[section]
\newtheorem{lemma}[theorem]{Lemma}

\newtheorem{proposition}[theorem]{Proposition}

\newtheorem{problem}[theorem]{Problem}
\theoremstyle{remark}
\newtheorem{remark}[theorem]{Remark}

\renewcommand{\div}{\operatorname{div}}
\newcommand{\curl}{\operatorname{curl}}

\newcommand{\dx}{\,dx}
\newcommand{\ds}{\,ds}
\newcommand{\id}{\operatorname{id}}

\newcommand{\Lag}[1]{\mathcal{L}_{#1}}
\newcommand{\Cle}[1]{\mathcal{Q}_{#1}}
\newcommand{\NedI}[1]{\mathcal{N}_{#1}^{\mathrm{I}}}
\newcommand{\NedII}[1]{\mathcal{N}_{#1}^{\mathrm{II}}}

\newcommand{\std}{\mathrm{std}}
\newcommand{\pr}{\mathrm{pr}}

\newcommand{\eremk}{\hbox{}\hfill\rule{0.8ex}{0.8ex}}
\newcommand{\dist}{\operatorname{dist}}

\usepackage{hyperref}
\usepackage{cleveref}

\newcommand{\commutingdiagram}{
	\begin{tikzpicture}[>=latex, node distance=3.5em and 4em]
		\node (tl) {$H^1(\Omega)$};
		\node (tr) [right=of tl] {$H(\curl,\Omega)$};
		\node (bl) [below=of tl] {$\Lag{k}(\mathcal{T}_h)$};
		\node (br) [below=of tr] {$\NedI{k-1}(\mathcal{T}_h)$};
		\draw[->] (tl) -- (tr) node[midway,above] {$\nabla$};
		\draw[->] (tl) -- (bl) node[midway,left] {$\Lag{k}$};
		\draw[->] (tr) -- (br) node[midway,right] {$\NedI{k-1}$};
		\draw[->] (bl) -- (br) node[midway,above] {$\nabla$};
	\end{tikzpicture}
}

\title[Pressure-robustness by commuting interpolation]{Pressure-robustness by commuting interpolation operators for Stokes discretizations with continuous pressures}

\author[P.~L.~Lederer]{Philip L. Lederer}
\address{Department of Mathematics, University of Hamburg, Bundesstra{\ss}e 55, D-20146 Hamburg, Germany}
\email{philip.lederer@uni-hamburg.de}

\author[T.~Vock]{Theresa Vock}
\address{Institute of Analysis and Scientific Computing, TU Wien, Wiedner Hauptstrasse 8--10, 1040 Wien, Austria}


\begin{document}

\begin{abstract}
Common finite element discretizations for the incompressible Stokes equations with a
continuous pressure approximation -- the MINI element, the Taylor--Hood element, or
stabilized equal-order elements -- are not pressure-robust: the velocity error is polluted
by the pressure best-approximation error, multiplied by the inverse viscosity. The
established remedy replaces the test function in the momentum equation by a
divergence-preserving reconstruction operator; for continuous pressure approximations this
construction is cumbersome, requiring the solution of local problems (Lederer, Linke,
Merdon \& Sch\"oberl, SIAM J. Numer. Anal., 2017). We propose a simpler alternative:
interpolating the force by a N\'ed\'elec interpolation operator of suitable order instead
of reconstructing the test function. Exploiting the commuting diagram property linking
N\'ed\'elec and Lagrange interpolation operators, we show that the resulting method is
pressure-robust for any $H^1$-conforming, continuous-pressure discretization. We give the
complete analysis for the MINI element, and then for stabilized equal-order $P^kP^k$
elements of arbitrary degree $k$, proving that the interpolation-induced consistency error
is pressure-robust and converges of optimal order. We validate the theory by several numerical examples.
\end{abstract}

\maketitle

\section{Introduction}
\label{sec:introduction}
Let $\Omega \subseteq \mathbb{R}^2$ be an open, simply connected, bounded set with smooth
boundary.  For a given force $f \in [L^2(\Omega)]^2$ and constant dynamic viscosity $\nu>0$, we consider the stationary Stokes equations
\begin{subequations}
	\label{eq:Stokes}
	\begin{alignat}{2}
	- \nu \, \Delta u  + \nabla p  &= f, &\quad &\text{in } \Omega, \label{eq:Stokes_line1} \\
	\div(u) &= 0, &\quad &\text{in } \Omega, \label{eq:Stokes_line2}
\end{alignat}
\end{subequations}
with homogeneous Dirichlet boundary conditions $u=0$ on $\partial\Omega$.  Here $u$ denotes the velocity and $p$ the pressure. Note that other boundary conditions are possible but we restrict ourselves to homogeneous Dirichlet boundary conditions for simplicity. Further, due to the already technical complexity of the analysis, we restrict ourselves to the two dimensional setting, but the results can be extended to three dimensions in a straightforward manner. To show the applicability we also present a numerical example in three dimensions.

Most classical finite element discretizations of \eqref{eq:Stokes}, see e.g. \cite{braess2013}, that are stable, convergent and easy to implement have the drawback that they are not pressure-robust.
This means that the velocity error depends on the approximation of the pressure, scaled by the inverse of the viscosity. In the case of a small viscosity this can lead to significant approximation errors and non-physical behavior of the discrete solution. This phenomenon originates in the relaxation of the divergence constraint in the discrete setting and has been well studied, see e.g. the overview article \cite{VolkerLinke2017}.

In recent years, pressure-robustness has become a major point of focus in the numerical
analysis of the incompressible (Navier--)Stokes equations, see e.g.\ \cite{BrenneckeLinkeMerdon2015,MR4615427,MR4235807,MR4161756,MR4250631}. Note that the literature is to big to be cited completely here, and we refer to the references in \cite{VolkerLinke2017} for a more complete overview.
Examples of pressure-robust discretizations are, e.g., based on finite elements providing exactly
divergence-free velocities, such as the Scott--Vogelius element \cite{SV85} or
$H(\div)$-conforming (hybrid) discontinuous Galerkin methods
\cite{MR2304270,LS16,MR4122492,MR3833698}. Another, more widely applicable approach is the
use of a velocity reconstruction operator, applied to test functions on the right hand side.
Especially in the case of continuous pressure approximations, the construction of a
reconstruction operator is cumbersome, since it involves solving several local problems
\cite{cont_rec}, in contrast to the case of discontinuous pressure approximations, where a local (most commonly $H(\div$)-conforming) reconstruction operator is applied locally as it was done in the pioneering work \cite{LinkeVariationalCrime}.

In this paper we present a novel method to achieve pressure-robustness, based on
interpolation operators, for finite elements for the Stokes equations with a
\emph{continuous} pressure approximation.
It combines the advantages of classical finite
element methods -- a simple implementation and optimal error convergence -- with pressure-robustness \textbf{without} the need of solving local (patch-wise) problems.
The method is based on the interpolation of the force term $f$ by a Nédélec interpolation operator
$\NedI{k-1}$ of first kind and of suitable order $k \in \mathbb{N}$.
Exploiting the commuting diagram property (for two space dimensions)
\begin{center}
	\commutingdiagram
\end{center}
where $\Lag{k}(\mathcal{T}_h)$ is a Lagrange finite element space of order $k$ and
$\Lag{k}$ the corresponding interpolation operator, we
show that replacing $(f,v_h)$ by $(\NedI{k-1} f, v_h)$ in the
right hand side of the discrete momentum equation results in a pressure-robust method. Note, that throughout the paper we do not
distinguish notationally between a finite element space and its associated interpolation
operator; which of the two is meant is always clear from the context, e.g.\ $\Lag{k}$
applied to a function versus $\Lag{k}(\mathcal{T}_h)$ or $v_h\in\Lag{k}(\mathcal{T}_h)$.
Although the insertion of an interpolation operator introduces a consistency error, we
prove that this error is not only independent of the pressure but converges at one order
higher than a naive estimate would suggest, so that optimal order of convergence with
respect to the mesh size $h$ is retained.

We develop the analysis first for the simplest possible setting, the MINI element (i.e.\
piecewise linear, cubic-bubble-enriched velocities together with piecewise linear
pressures), where the method requires only the lowest-order Nédélec interpolation
operator. We then turn to the case of the family of
stabilized, equal-order $P^kP^k$ elements for arbitrary $k \geq 1$, where a residual-based
stabilization is needed for stability in the first place. We prove that the
interpolation-based consistency error remains pressure-independent -- and in fact
superconverges -- for every polynomial degree $k$. As a byproduct we also comment on the
Taylor--Hood element, which shows the limitations of the present technique: pressure
robustness is retained, but the loss of one order of approximation in the interpolation
operator (which has to be chosen to commute with the \emph{pressure} space, one degree
below the velocity space) prevents optimal convergence in the $L^2$ norm. Note, that the energy-norm error still converges optimally.

This paper is organized as follows. \Cref{sec:preliminaries} collects the necessary
background: the variational theory of the Stokes equations and of mixed finite elements,
the mechanism behind the lack of pressure-robustness and the classical remedy by
reconstruction operators, the de Rham complex together with $H(\curl)$-conforming Nédélec
finite elements, and -- as the main technical tool for everything that follows -- a
proposition showing that the interpolation-based consistency error is superconvergent. \Cref{sec:method} introduces the novel method. \Cref{sec:mini} carries
out the complete error analysis for the MINI element. \Cref{sec:stabpkpk} treats the
general stabilized $P^kP^k$ element and contains our main pressure-robustness and
optimal-order convergence result for arbitrary polynomial degree. \Cref{sec:taylorhood}
briefly discusses the Taylor--Hood element as a cautionary example, and \Cref{sec:discussion_regularity} discusses the regularity assumptions needed for the analysis and their implications for the convergence rates.
Finally, \Cref{sec:numerics} presents numerical experiments that confirm the theoretical findings.

\section{Preliminaries}
\label{sec:preliminaries}

\subsection{The weak formulation of the Stokes equations}
\label{ssec:stokes}
Multiplying \eqref{eq:Stokes_line1} by
$v \in V \coloneqq [H^1_0(\Omega)]^2$ and \eqref{eq:Stokes_line2} by $q \in
Q\coloneqq L^2_0(\Omega)$, where $L^2_0(\Omega)$ is the space of functions in $L^2(\Omega)$ with zero mean, gives, after integration by parts, the variational problem: find $(u,p)\in V\times Q$ such
that
\begin{subequations}\label{eq:weak_stokes}
\begin{alignat}{2}
	\nu \, a(u,v)+b(v,p)&=(f,v)_{L^2} && \quad \forall v \in V,\\
	b(u,q)&=0 && \quad \forall q \in Q,
\end{alignat}
\end{subequations}
with the bilinear forms $a(u,v) \coloneqq (\nabla u, \nabla v)_{L^2}$ and
$b(u,p)\coloneqq -(\div(u),p)_{L^2}$. Note, that throughout the paper we omit the domain from norms, seminorms and inner products
taken over the whole set $\Omega$, e.g.\ $\|\cdot\|_{H^1}\coloneqq\|\cdot\|_{H^1(\Omega)}$,
$|\cdot|_{H^j}\coloneqq|\cdot|_{H^j(\Omega)}$ and $(\cdot,\cdot)_{L^2}\coloneqq(\cdot,\cdot)_{L^2(\Omega)}$,
and only write the domain explicitly when it is a proper subset of $\Omega$, such as an
element $T$, an edge patch $\omega_E$ or the boundary $\partial\Omega$. Problem \eqref{eq:weak_stokes} has the
structure of a saddle point problem, the pressure playing the role of a Lagrange multiplier
for the incompressibility constraint. By the standard theory of saddle-point problems,
see  \cite{brezzi1974, braess2013,BoffiBrezziFortin}, it is well posed provided that $a(\cdot,\cdot)$ is coercive on the
kernel $V_0 \coloneqq \{v \in V : b(v,q)=0 \; \forall q\in Q\}$ and that $b(\cdot,\cdot)$
satisfies the inf-sup (LBB) condition
\begin{align*}
	\sup_{v \in V\backslash \{0\}} \frac{b(v,q)}{\|v\|_{H^1}} \geq \beta \, \|q\|_{L^2}
	\qquad \forall q \in Q,
\end{align*}
with a constant $\beta>0$ independent of $\nu$; both properties hold for the Stokes problem
\cite{GiraultRaviart86}. Later we use the Aubin--Nitsche technique
\cite{aubin1967,NITSCHE1968} to derive $L^2$ error estimates for the velocity, which
requires the \emph{dual} Stokes problem: for $e \in L^2(\Omega)$, find $(w,r)\in V\times Q$
such that
\begin{subequations}	\label{eq:dual_stokes}
\begin{alignat}{2}
	\nu \, a(v,w)+b(v,r)&=(e,v) && \quad \forall v \in V,\\
	b(w,q)&=0 && \quad \forall q \in Q.
\end{alignat}
\end{subequations}

\subsection{Mixed finite elements and the classical error estimate}
\label{ssec:mixedfem}
Let $\mathcal{T}_h$ be a shape-regular, quasi-uniform triangulation of $\Omega$ of mesh size
$h = \max_{T\in\mathcal{T}_h} h_T$, with edge set $\mathcal{E}_h$, and let $V_h \subset V$, $Q_h \subset Q$ be conforming
finite element spaces. We abbreviate $\sum_T \coloneqq \sum_{T\in\mathcal{T}_h}$ and
$\sum_E \coloneqq \sum_{E\in\mathcal{E}_h}$; sums restricted to a proper subset of edges
(e.g.\ $\mathcal{E}_h^{\mathrm o}$ - the set of interior edges, $\mathcal{E}_h^\partial$ - the set of boundary edges) are always written with the
subset given explicitly. The
discrete Stokes problem reads: find $(u_h,p_h)\in V_h\times Q_h$ such that
\begin{alignat}{2}
	\label{eq:disc_stokes}
	\nu \, a(u_h,v_h)+b(v_h,p_h)&=(f,v_h) && \quad \forall v_h \in V_h,\\
	b(u_h,q_h)&=0 && \quad \forall q_h \in Q_h, \notag
\end{alignat}
and is well posed provided the discrete LBB-condition
\begin{align}
	\label{eq:discreteLBB}
	\sup_{v_h \in V_h\backslash \{0\}} \frac{b(v_h,q_h)}{\|v_h\|_{H^1}} \geq \beta_h \, \|q_h\|_{L^2}
	\qquad \forall q_h \in Q_h,
\end{align}
holds for some $\beta_h>0$; coercivity is inherited from the continuous problem since
$V_h\subset V$ and $a(\cdot,\cdot)$ is coercive on $V$ (and not only on the kernel $V_0$). If \eqref{eq:discreteLBB} holds, the classical (Céa-type) error estimate
\begin{align}
	\label{eq:classical_error}
	\|u-u_h\|_{H^1} \leq C \left( \inf_{v_h \in V_h} \|u-v_h\|_{H^1}
	+ \frac{1}{\nu} \inf_{q_h \in Q_h}\|p-q_h\|_{L^2} \right),
\end{align}
holds with a constant $C$ depending only on the continuity/coercivity constants of
$a(\cdot,\cdot)$, the continuity constant of $b(\cdot,\cdot)$ and $\beta_h^{-1}$, but not on
$\nu$; see \cite{BrezziFortinMixedAndHybridFEM91,GiraultRaviart86} for a proof. From
here on we write $A \lesssim B$ for $A \leq CB$ with a constant $C>0$ that depends only on
the shape-regularity of $\mathcal{T}_h$, the domain $\Omega$ and, where applicable, the
polynomial degree $k$ and the stabilization parameter $\alpha$ -- but \emph{never} on the
mesh size $h$ or the viscosity $\nu$.

\begin{remark}[Kernel inclusion]
	\label{rem:kernel_inclusion}
	If the discrete kernel
	\begin{align} \label{eq:discrete_kernel}
		V_{h,0}\coloneqq\{v_h\in V_h : b(v_h,q_h)=0\ \forall q_h\in Q_h\},
	\end{align}
	is a subset of the continuous kernel $V_0$, the pressure term drops out of
	\eqref{eq:classical_error} entirely and the velocity error becomes a genuine
	best-approximation estimate,
	$\|u-u_h\|_{H^1} \leq C \inf_{v_h \in V_h} \|u-v_h\|_{H^1}$.
	This kernel-inclusion property, however, fails to hold for essentially all commonly used
	$H^1$-conforming finite element pairs \cite{VolkerLinke2017}, except for some special cases, e.g. the Scott-Vogelius elements, which is the root cause of the loss of pressure-robustness discussed next.
	\eremk
\end{remark}

\subsection{Pressure-robustness and reconstruction operators}
\label{ssec:pressure_robustness}
Estimate \eqref{eq:classical_error} shows that the velocity error is polluted by the
best-approximation error of the pressure, multiplied by $1/\nu$. For small viscosities this
term can dominate the error entirely. We call a method \emph{pressure-robust} if the
velocity error is independent of the error of the pressure.

To understand the origin of this phenomenon, following \cite{no-flow-example1997} consider
a force given as a gradient field, $f = \nabla \Phi$. Then the continuous momentum equation
reads $\nu (\nabla u,\nabla v) - (\div(v),p) = (\nabla \Phi, v)$ for all $v\in V$; after
integration by parts this becomes $\nu (\nabla u,\nabla v)+(\div(v),\Phi-p)=0$, which is
satisfied by $u=0$, $p=\Phi$, and by uniqueness this is \emph{the} solution: gradient
forces only affect the pressure, the velocity remains unchanged. In the discrete setting,
testing whether $(0,\Pi_{Q_h}\Phi)$ (with $\Pi_{Q_h}$ the $L^2$-projection onto $Q_h$)
solves \eqref{eq:disc_stokes} leads, for a discretely divergence-free testfunction  $v_h \in V_{h,0}$, to the requirement
$(\div(v_h),\Phi)=0$. Since neither $\Phi\in Q_h$ nor $v_h \in V_0$ hold in general, this
term does not vanish unless $V_{h,0}\subset V_0$ (\Cref{rem:kernel_inclusion}): gradient
forces on the right hand side pollute \emph{both} the discrete pressure \emph{and} the
discrete velocity.

The established remedy is the use of a \emph{reconstruction operator} $\mathcal{R}$ applied
to the test function, replacing $(f,v_h)$ by $(f,\mathcal{R}v_h)$ in \eqref{eq:disc_stokes}
\cite{LinkeVariationalCrime}. If $\mathcal{R}$ maps discretely divergence-free functions to
exactly divergence-free ones and is close to the identity in the sense that
$\|v-\mathcal{R}v\|_{L^2(T)} \lesssim h^m \|v\|_{H^m(T)}$, pressure-robustness is
restored while the stiffness matrix -- and hence most of an existing implementation --
remains unchanged \cite{MR3460110}. Such operators have been constructed, e.g., for the
Crouzeix--Raviart element \cite{LinkeVariationalCrime}, for the $P^{2+}P^{1,\mathrm{disc}}$
element \cite{VolkerLinke2017}, and, at the price of solving local problems, for the MINI
element and the family of Taylor--Hood elements with \emph{continuous} pressure
approximation \cite{cont_rec}. It is exactly this last, cumbersome, construction that the
present paper avoids.

\subsection{$H(\curl)$, Nédélec elements, and the commuting diagram}
\label{ssec:nedelec}
Our construction relies on interpolation operators into $H(\curl)$-conforming finite
element spaces. For $\Omega\subset\mathbb{R}^2$, let
\begin{align*}
	H(\curl, \Omega) \coloneqq \{ u \in [L^2(\Omega)]^2 \; : \; \curl u \in L^2(\Omega) \},
	\qquad \curl(u) \coloneqq \partial_x u_2 - \partial_y u_1,
\end{align*}
equipped with the norm $\|u\|_{\curl}^2 \coloneqq \|u\|_{L^2}^2+\|\curl u\|_{L^2}^2$. As
$\curl u = -\div(u^\perp)$, the space is isomorphic to $H(\div)$ by a rotation of
$\pi/2$; in particular there is a well-defined tangential trace $(\cdot)_t : H(\curl)\to
H^{-1/2}(\Gamma)$, and $H(\curl)$-conformity is equivalent to continuity of this tangential
trace across element interfaces \cite{Nedelec1980,GiraultRaviart86}. The relevant finite
elements are the Nédélec elements, which we construct following the hierarchical,
exact-sequence-compatible approach of \cite{zaglmayr2005,szthesis2006} rather than
Nédélec's original construction \cite{Nedelec1980}; we still use his
terminology of \emph{first} and \emph{second kind}.

In the following we use the notation $P^k(T)$ for the space of polynomials of degree at most $k$ on a triangle $T$, and $[P^k(T)]^2$ for the vector-valued version.

\paragraph{Lowest order elements} On a triangle $T$ with vertices $v_0,v_1,v_2$, edges
$E_i$ (opposite $v_i$) and barycentric coordinates $\lambda_i \in P^1(T)$, the lowest order Nédélec
space of first kind is
\begin{align*}
	\NedI{0}(T) \coloneqq \{ a+c(y,-x)^T : a\in\mathbb{R}^2,\ c\in\mathbb{R}\},
	\qquad [P^0(T)]^2 \subset \NedI{0}(T) \subset [P^1(T)]^2,
\end{align*}
with degrees of freedom $v \mapsto \int_E v\cdot\tau_E\,\ds$, $\tau_E$ the tangent of $E$,
and dual basis (edge shape functions) $\varphi_E = \lambda_0\nabla\lambda_1 -
\nabla\lambda_0\,\lambda_1$ for $E=(v_0,v_1)$. The associated \emph{lowest order Nédélec
interpolation operator of first kind} is then given by
$\NedI{0} v \coloneqq \sum_E \big(\int_E
v\cdot\tau_E\,\ds\big)\varphi_E$.
Enlarging the degrees of freedom to
$\int_E v\cdot\tau_E\,p\,\ds$ for $p\in P^1(E)$ produces the lowest order Nédélec space of
\emph{second kind}, $\NedII{1}(T) = [P^1(T)]^2$, with dual basis
$\varphi_E^0$ (as above) and $\varphi_E^1 = \nabla b_E$, where $b_E = \lambda_0 \lambda_1$ denotes the quadratic
edge bubble of $E = (v_0, v_1)$; the corresponding operator $\NedII{1}$
satisfies $\|u-\NedII{1} u\|_{L^2} \lesssim
h^2|u|_{H^2}$. As $\varphi_E^0$ is shared by both operators, their difference
localizes to gradients of edge bubbles,
\begin{align}
	\label{eq:diff_ned_op_lowest}
	\NedII{1} f - \NedI{0} f
	= \sum_E c_E \nabla b_E, \qquad
	c_E \coloneqq \int_E f\cdot\tau_E\, p_E \,\ds, \quad p_E \in P^1(E)  \,/\, \mathbb{R},
\end{align}
and a duality argument together with the inverse inequality gives -- for smooth enough functions $f$ -- the coefficient bound
\begin{align}
	\label{eq:cE_bound}
	|c_E| \lesssim h\, |f|_{H^1(T_E)},
\end{align}
for a triangle $T_E$ incident to $E$; see \cite{szthesis2006}.

\paragraph{Higher order elements} For $k>0$ the spaces $\NedI{k}$ (of
dimension $k^2+4k+3$, containing $[P^k]^2$ enriched by curl-generating polynomials of one
higher degree) and $\NedII{k} = [P^k]^2$ are built from explicit
hierarchical shape functions in terms of (scaled, integrated) Legendre polynomials, split
into edge-based and element-based contributions, with degrees of freedom of moment type on
edges and elements; the full, index-heavy construction can be found in
\cite{zaglmayr2005,szthesis2006}. They satisfy the inclusion
\begin{align*}
	[P^0]^2 \subset \NedI{0}, \qquad \NedI{j-1} \subset \NedII{j} = [P^j]^2 \subset \NedI{j}
	\qquad \text{for every } j\geq 1,
\end{align*}
and their interpolation operators satisfy the estimates
\begin{align}
	\label{eq:nedelec_general_error}
	\|u - \NedI{k} u \|_{H^m} \lesssim h^{j-m} |u|_{H^j},
	\qquad
	\|u - \NedII{k} u\|_{L^2} \lesssim h^{k+1} |u|_{H^{k+1}},
\end{align}
for $0 \leq m \leq j \leq k+1$ and $j \ge 1$; see \cite{BoffiBrezziFortin} and
\cite[Ch.~4]{szthesis2006}. As in the lowest order case, all basis functions of
$\NedI{k}$ are shared by $\NedII{k+1}$ except for the
highest order edge shape functions and a family of element shape functions, which are all
of the form $\nabla(\text{polynomial})$, thus have a vanishing $\curl$; consequently
\begin{align}
	\label{eq:diff_ned_op_general}
	\NedII{k+1} f - \NedI{k} f
	= \sum_E \alpha_E \varphi_{k+1}^E + \sum_T\sum_{r=0}^{k-1}\beta_T^r\,\psi^T_{r,k-1-r},
\end{align}
with $\varphi_{k+1}^E$ supported on the edge patch $\omega_E$ and $\psi^T_{r,s}$ supported
on $T$, both of gradient type, and coefficients $\alpha_E \coloneqq \int_E f\cdot\tau_E\,
p_E^{k+1}\,\ds$, $\beta_T^r \coloneqq \int_T f\cdot q_T^r \dx$
(for suitably scaled
$p_E^{k+1}\in P^{k+1}(E)$, $q_T^r\in\nabla P^{r+1}_{0}(T)$, where $P^{r+1}_{0}(T)$ denotes the space  of polynomials of degree at most $r+1$ with vanishing trace on $\partial T$) obeying, exactly as in
\eqref{eq:cE_bound},
\begin{align}
	\label{eq:alphabeta_bound}
	|\alpha_E| \lesssim h^{k+1}|f|_{H^{k+1}(T_E)}, \qquad |\beta_T^r| \lesssim h^{k+1}|f|_{H^{k+1}(T)};
\end{align}
we again refer to \cite{szthesis2006} for the (elementary but lengthy) verification, which
proceeds by duality exactly as for \eqref{eq:cE_bound}.

Finally, the connection between the Lagrange spaces $\Lag{k}(\mathcal{T}_h)$ (built
from the same hierarchical family so as to be exact-sequence compatible with the Nédélec
spaces, see \cite{zaglmayr2005,szthesis2006}) and the Nédélec spaces is given by the
discrete de Rham complex and the commuting diagram property.

\begin{theorem}[de Rham complex]
	\label{thm:deRham}
	Let $\Omega\subset\mathbb{R}^2$ be a bounded, simply-connected Lipschitz domain. The
	sequence
	$\mathbb{R} \xrightarrow{\mathrm{id}} H^1(\Omega) \xrightarrow{\nabla} H(\curl,\Omega)
	\xrightarrow{\curl} L^2(\Omega)$
	is exact, and so is its discrete counterpart
	$\Lag{k}(\mathcal{T}_h) \xrightarrow{\nabla} \NedI{k-1}(\mathcal{T}_h)$;
	see \cite{szthesis2006}.
\end{theorem}

\begin{theorem}[Commuting diagram]
	\label{thm:commuting}
	Under the assumptions of \Cref{thm:deRham}, the diagram
	\begin{center}
		\commutingdiagram
	\end{center}
	commutes, i.e.\ $\NedI{k-1} (\nabla \Phi) = \nabla(\Lag{k} \Phi)$ for all $\Phi \in H^2(\Omega)$;
	see \cite{szthesis2006}.
\end{theorem}

\begin{remark}
	The assumption that $\Phi \in H^2(\Omega)$, rather than merely $H^1(\Omega)$, is what makes the
	nodal degrees of freedom of $\Lag{k}$ -- present already for the vertex functions at every
	order $k\ge1$ -- classically well-defined, via the Sobolev embedding $H^2(\Omega)\subset
	C^0(\bar\Omega)$ in two dimensions; $H^1(\Omega)$ alone is the borderline case and does
	\emph{not} embed into $C^0(\bar\Omega)$.
	\eremk
\end{remark}

We also record the classical Lagrange interpolation error estimate,
\begin{align}
	\label{eq:lagrange_error}
	\|u - \Lag{k} u\|_{H^m} \lesssim h^{j-m}|u|_{H^j}, \qquad 0\leq m \leq j \leq k+1, \ \ j \ge 2,
\end{align}
as well as the analogous estimate for the $L^2$-orthogonal
projection $\Pi_\omega^k$ onto $P^k(\omega)$,
\begin{align}
	\label{eq:l2proj_error}
	\|u-\Pi_\omega^k u\|_{H^m(\omega)} \lesssim h^{j-m}|u|_{H^j(\omega)}, \qquad 0\leq m \leq j \leq k+1;
\end{align}
see \cite{ern2021finite}.

Additionally, we also need an $H^1$-stable Cl\'ement-like quasi interpolation operator $\Cle{1}: H^1(\Omega) \to \Lag{1}(\mathcal{T}_h)$ that preserves (homogeneous) boundary data, and satisfies the following approximation properties \cite{Clement1975,ScottZhang1990,ern2021finite}:
\begin{align}
	\label{eq:quasi_interpolation}
	\|u-\Cle{1} u\|_{L^2(T)} + h \|u-\Cle{1} u\|_{H^1(T)} \lesssim h^j |u|_{H^j(\omega_T)}, \qquad 1\leq j \leq 2,
\end{align}
for all $T\in\mathcal{T}_h$, where $\omega_T$ is the patch of elements sharing a vertex with $T$. Particularly, this induces the stability property
\begin{align}
	\label{eq:quasi_interpolation_stability}
	\|\Cle{1} u\|_{H^1} \lesssim \|u\|_{H^1}.
\end{align}

Note that in order to simplify the notation, we do not distinguish between vector valued and scalar-valued version of the interpolation operators $\Lag{k}$ and $\Cle{1}$, as the context will always make it clear which one is meant.

\subsection{A superconvergent consistency error}
\label{ssec:consistency}
This subsection contains the key technical ingredient of the paper: an estimate for the
quantity $(f-\NedI{k} f, g)_{L^2}$ that converges one order faster than the naive Cauchy--Schwarz
bound $\|f-\NedI{k} f\|_{L^2}\|g\|_{L^2} \lesssim
h^{k+1}\|f\|_{H^{k+1}}\|g\|_{L^2}$ would suggest. It is precisely this superconvergence that
allows the interpolated method to retain the optimal order of convergence (of the $L^2$-norm error) despite the
consistency error introduced by the interpolation of the force.

We first record two auxiliary results. Recall that $\mathcal{E}_h = \mathcal{E}_h^{\mathrm o}
\, \cup\, \mathcal{E}_h^\partial$ splits into inner and boundary edges, respectively, and that the
edge patch of a boundary edge consists of a single triangle.

\begin{lemma}[Orthogonality of edge bubbles]
	\label{lem:orthogonality}
	Let $E\in\mathcal{E}_h^{\mathrm o}$ and let $\nabla b_E$ be the associated edge bubble
	gradient. Then $(\nabla b_E, \xi)_{L^2(\omega_E)} = 0$ for every constant vector
	$\xi\in\mathbb{R}^2$.
\end{lemma}
\begin{proof}
	Integration by parts gives
	\begin{align*}
		(\nabla b_E,\xi)_{L^2(\omega_E)} = (b_E n, \xi)_{L^2(\partial\omega_E)} - (b_E,\div(\xi))_{L^2(\omega_E)} = 0,
	\end{align*}
	since $b_E$ vanishes on $\partial\omega_E$ and $\xi$ is constant.
\end{proof}

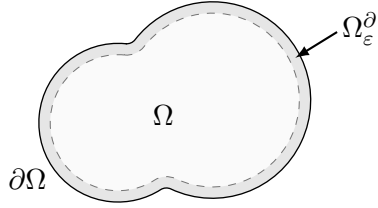
\begin{figure}[h]
	\centering
	\begin{tikzpicture}[scale=1]
		\draw[line width=0.5pt, fill=black!10]
			(84.517:1.050) arc (84.517:302.475:1.050)
			arc (122.475:64.891:0.140)
			arc (244.891:502.101:1.300)
			arc (322.101:264.517:0.140) -- cycle;
		\draw[line width=0.3pt, fill=black!2, draw=black!55, dashed]
			(84.517:0.900) arc (84.517:302.475:0.900)
			arc (122.475:64.891:0.290)
			arc (244.891:502.101:1.150)
			arc (322.101:264.517:0.290) -- cycle;
		\node at (0.6,0.1) {$\Omega$};
		\node at (3.2,1.25) {$\Omega_\varepsilon^\partial$};
		\draw[->,>=latex,shorten >=1.5pt, thick] (2.9,1.2) -- (2.30,0.83);
		\node at (-1.2,-0.75) {$\partial\Omega$};
	\end{tikzpicture}
	\caption{Tubular neighborhood $\Omega_\varepsilon^\partial$ (dark grey) of thickness $\varepsilon$ of the
	boundary $\partial\Omega$ of the domain $\Omega$ (light grey).}
	\label{fig:tubular_neighborhood}
\end{figure}
The following Lemma has already been used in \cite{Faustmann_2015} and \cite{Melenk2010} to estimate boundary layer contributions in the context of finite element methods. For this we define the tubular neighborhood of the boundary $\partial\Omega$ of thickness $\varepsilon>0$, see
	\Cref{fig:tubular_neighborhood}, as
\begin{align*}
	\Omega^\partial_\varepsilon \coloneqq \{x\in\Omega : \dist(x,\partial\Omega)<\varepsilon\}.
\end{align*}

\begin{lemma}[Boundary layer estimate, {\cite{Faustmann_2015,Melenk2010}}]
	\label{lem:boundary_layer}
	Let $f\in H^1(\Omega)$ and let $\Omega^\partial_\varepsilon$ denote the tubular
	neighborhood of $\partial\Omega$ of thickness $\varepsilon>0$. Then
	\begin{align*}
		\|f\|_{L^2(\Omega^\partial_\varepsilon)} \lesssim \sqrt{\varepsilon}\,\|f\|_{L^2(\partial\Omega)}
	+ \varepsilon \|\nabla f\|_{L^2(\Omega^\partial_\varepsilon)}.
	\end{align*}
	For the particular choice $\varepsilon = h$, this gives by the trace theorem
	\begin{align*}
		\|f\|_{L^2(\Omega^\partial_h)} \lesssim \sqrt{h}\,\|f\|_{L^2(\partial\Omega)}
	+ h \|\nabla f\|_{L^2(\Omega^\partial_h)} \lesssim \sqrt{h} \|f\|_{H^1(\Omega)}.
	\end{align*}
\end{lemma}

\begin{proposition}
	\label{prop:consistency_low}
	Let $f\in H^2(\Omega)$ and $g\in H^1(\Omega)$. Then
	\begin{align}
		\label{eq:consistency_low}
		(f-\NedI{0} f, g)_{L^2} \lesssim h^2 \|f\|_{H^2}\|g\|_{H^1}.
	\end{align}
\end{proposition}
\begin{proof}
	Adding and subtracting $\NedII{1} f$ we get
	\begin{align*}
		(f-\NedI{0} f, g)_{L^2}
		= (f-\NedII{1} f, g)_{L^2}
		+ (\NedII{1} f-\NedI{0} f, g)_{L^2}.
	\end{align*}
	The first term is bounded directly by Cauchy--Schwarz and \eqref{eq:nedelec_general_error} giving
	\begin{align*}
	(f-\NedII{1} f, g)_{L^2} \leq
	\|f-\NedII{1} f\|_{L^2}\|g\|_{L^2} \lesssim h^2|f|_{H^2}\|g\|_{L^2}.
	\end{align*}
	For the second term we use the localized representation \eqref{eq:diff_ned_op_lowest} and
	split the sum over edges into $\mathcal{E}_h^{\mathrm o}$ and $\mathcal{E}_h^\partial$. On
	inner edges, \Cref{lem:orthogonality} allows us to insert the $L^2(\omega_E)$-projection
	$\Pi^0_{\omega_E} g$ of $g$ onto constants at no cost,
	\begin{align*}
		\sum_{E\in\mathcal{E}_h^{\mathrm o}} c_E (\nabla b_E, g)_{L^2(\omega_E)}
		&= \sum_{E\in\mathcal{E}_h^{\mathrm o}} c_E (\nabla b_E, g-\Pi^0_{\omega_E}g)_{L^2(\omega_E)}\\
		&\overset{\text{C.S.}}{\leq}
		\sum_{E\in\mathcal{E}_h^{\mathrm o}} |c_E|\,\|\nabla b_E\|_{L^2(\omega_E)}\,\|g-\Pi^0_{\omega_E}g\|_{L^2(\omega_E)}\\
		&\overset{\eqref{eq:l2proj_error}}{\lesssim} h\,|g|_{H^1}\sum_{E\in\mathcal{E}_h^{\mathrm o}} |c_E|\,\|\nabla b_E\|_{L^2(\omega_E)}
		\overset{\eqref{eq:cE_bound}}{\lesssim} h^2\,|g|_{H^1}\,|f|_{H^1},
	\end{align*}
	using that $\|\nabla b_E\|_{L^2(\omega_E)}$ scales like a constant and that each triangle
	belongs to at most three patches $\omega_E$. On boundary edges, where
	\Cref{lem:orthogonality} is not available (the patch $\omega_E$ degenerates to the single
	triangle $T_E$), we argue instead via \eqref{eq:cE_bound} and \Cref{lem:boundary_layer}
	with $\varepsilon=h$,
	\begin{align*}
		\sum_{E\in\mathcal{E}_h^\partial} c_E (\nabla b_E,g)_{L^2(T_E)}
		&\overset{\text{C.S.}}{\leq} \sum_{E\in\mathcal{E}_h^\partial} |c_E|\,\|\nabla b_E\|_{L^2(T_E)}\,\|g\|_{L^2(T_E)}\\
		&\overset{\eqref{eq:cE_bound}}{\lesssim}
		h\,|f|_{H^1(\Omega_h^\partial)}\,\|g\|_{L^2(\Omega_h^\partial)}
		\overset{\text{\Cref{lem:boundary_layer}}}{\lesssim}
		h^2\|f\|_{H^2}\|g\|_{H^1}.
	\end{align*}
	Combining both estimates and neglecting higher order terms gives \eqref{eq:consistency_low}.
\end{proof}

The same argument, applied to the higher-order difference formula
\eqref{eq:diff_ned_op_general} and the coefficient bounds \eqref{eq:alphabeta_bound} instead
of \eqref{eq:diff_ned_op_lowest} and \eqref{eq:cE_bound}, yields the following
generalization;

\begin{proposition}
	\label{prop:consistency_high}
	Let $k\in\mathbb{N}_0$, $f\in H^{k+2}(\Omega)$ and $g\in H^1(\Omega)$. Then
	\begin{align}
		\label{eq:consistency_high}
		(f-\NedI{k} f, g)_{L^2} \lesssim h^{k+2} \|f\|_{H^{k+2}}\|g\|_{H^1}.
	\end{align}
\end{proposition}
\begin{proof}
	For $k=0$ this is \Cref{prop:consistency_low}. For $k\geq 1$ one proceeds exactly
	as in the proof of \Cref{prop:consistency_low}, now with $f-\NedI{k} f$
	split via $\NedII{k+1} f$ and
	\eqref{eq:diff_ned_op_general}: the first (non-localized) part converges at order
	$h^{k+2}$ by \eqref{eq:nedelec_general_error}; for the localized part, the same
	integration-by-parts computation as in \Cref{lem:orthogonality} shows that the highest
	order edge and element shape functions $\varphi_{k+1}^E$, $\psi^T_{r,s}$ in
	\eqref{eq:diff_ned_op_general} -- being themselves gradients of scaled Legendre-type
	polynomials that vanish identically on the boundary of their support -- are $L^2$-orthogonal
	to constants on $\omega_E$, respectively $T$, so that a copy of $g-\Pi^0 g$ can again be
	inserted at the cost of one extra power of $h$ by \eqref{eq:l2proj_error}; combined with
	\eqref{eq:alphabeta_bound} this gives the order $h^{k+2}$ on inner edges and interior
	elements, and \Cref{lem:boundary_layer} handles the boundary edges as before.
\end{proof}

\begin{remark}
	\label{rem:h1_consistency}
	In the $H^1$-norm error analysis below we only ever need the cruder, non-superconvergent
	bound
	\begin{align*}
		(\NedI{k}(\Delta u)-\Delta u, v_h) \leq
		\|\NedI{k}(\Delta u)-\Delta u\|_{L^2}\|v_h\|_{L^2}
		\lesssim h^{k+1}|\Delta u|_{H^{k+1}}\|v_h\|_{L^2},
	\end{align*}
	which follows directly
	from \eqref{eq:nedelec_general_error} (with $m=0$, $j=k+1$) by Cauchy--Schwarz and requires less regularity than
	\Cref{prop:consistency_high}. It is only for the $L^2$-norm estimate, where the second
	argument is itself an interpolant of a smooth dual solution, that the extra order gained in
	\Cref{prop:consistency_high} is essential.
	\eremk
\end{remark}

\section{Pressure-robustness by commuting interpolation operators}
\label{sec:method}
As already mentioned in \Cref{ssec:pressure_robustness}, a reconstruction operator applied to the test function is elaborate to construct for a
continuous pressure approximation, see \cite{cont_rec}. Instead, we map the force $f$ by an appropriate
interpolation operator $\NedI{k-1}$: the right hand side $(f,v_h)$ in the momentum equation
becomes $(\NedI{k-1}f,v_h)$. The left hand side of the equation is not altered, so the same
techniques can be used for assembling the stiffness matrix; the only price to pay is a
consistency error, controlled in \Cref{ssec:consistency}, which turns out not to limit the
order of convergence for any of the elements discussed below.

To motivate the modified right hand side, recall from \Cref{ssec:pressure_robustness} that a
gradient force $f=\nabla\Phi$ is exactly balanced by the pressure on the continuous level,
$(u,p)=(0,\Phi)$, but that this balance is destroyed in the discrete setting because
$(\div(v_h),\Phi) \neq (\div(v_h),\Pi_{Q_h}\Phi)$ in general. Choosing the Nédélec interpolation
operator that commutes, via \Cref{thm:commuting}, with the Lagrange interpolation operator
$\Lag{k}$ of the discrete pressure space repairs exactly this defect. Since
\begin{align*}
	\NedI{k-1} f = \NedI{k-1} (\nabla\Phi)
	\overset{\text{\Cref{thm:commuting}}}{=} \nabla(\Lag{k} \Phi),
\end{align*}
repeating the computation of \Cref{ssec:pressure_robustness} with the modified right hand side, the
discrete momentum equation $\nu(\nabla u_h,\nabla v_h) - (\div(v_h),p_h) =
(\nabla(\Lag{k}\Phi),v_h)$ becomes, after integrating by parts,
$\nu(\nabla u_h,\nabla v_h)-(\div(v_h),p_h) = (\div(v_h),\Lag{k}\Phi)$,
and the right hand side now vanishes exactly for $v_h \in V_{h,0}$, because
$\Lag{k}\Phi \in Q_h$ \emph{by construction}. Gradient forces are
therefore entirely absorbed by the pressure, and we can expect the resulting method to be
pressure-robust. The remainder of the paper carries out the corresponding error analysis
first for the MINI element (\Cref{sec:mini}), and then for the
stabilized $P^kP^k$ element (\Cref{sec:stabpkpk}) for arbitrary $k\geq1$.

\section{The MINI element}
\label{sec:mini}
The MINI element \cite{MINI1984} approximates the pressure by continuous, piecewise linear
functions,
\begin{align*}
	Q_h \coloneqq \{ q \in C(\Omega) : q|_T \in P^1(T)\ \forall T\in\mathcal{T}_h\} \cap Q = \Lag{1}(\mathcal{T}_h) \cap Q,
\end{align*}
and enriches the space of continuous, piecewise linear vector fields by local cubic
bubbles, $B\coloneqq \{b : b|_T \in P^3(T)\cap H^1_0(T)\ \forall T\in\mathcal{T}_h\}$, to
restore stability:
\begin{align*}
	V_h \coloneqq \big( [\Lag{1}(\mathcal{T}_h)]^2 \oplus [B]^2 \big) \cap V.
\end{align*}
It satisfies the discrete LBB-condition \eqref{eq:discreteLBB}, and, denoting
$B_h((u_h,p_h),(v_h,q_h)) \coloneqq \nu\,a(u_h,v_h)+b(v_h,p_h)+b(u_h,q_h)$, the
inf-sup estimate
\begin{align}
	\label{eq:MINI_infsup}
	\|u_h\|_{H^1} + \frac{1}{\nu}\|p_h\|_{L^2} &\lesssim
	\underset{(v_h,q_h)\neq(0,0)}{\sup_{(v_h,q_h)\in V_h\times Q_h}} \frac{B_h((u_h,p_h),(v_h,q_h))}{\nu\|v_h\|_{H^1}+\|q_h\|_{L^2}}\\
	&\qquad \forall (u_h,p_h)\in V_h\times Q_h, \notag
\end{align}
see \cite{MINI1984,BrezziFortinMixedAndHybridFEM91}.
Since the pressure is approximated by piecewise linear functions, the commuting diagram (\Cref{thm:commuting}) suggests the lowest order Nédélec operator $\NedI{0}$ for the force term to define the modified discrete problem:

\begin{problem}
	\label{prob:mini}
	Find $(u_h,p_h)\in V_h\times Q_h$ such that
	\begin{alignat*}{2}
		\nu\,a(u_h,v_h)+b(v_h,p_h) &= (\NedI{0} f, v_h) && \quad \forall v_h\in V_h,\\
		b(u_h,q_h) &= 0 && \quad \forall q_h \in Q_h.
	\end{alignat*}
\end{problem}

In a first step we will provide a best approximation result which takes into account the consistency error introduced by the interpolation of the force. Recall that the dual norm of a functional $f\in V_h^*$ is defined as
\begin{align*}
	\|f\|_{V_h^*} \coloneqq \sup_{v_h\in V_h\setminus\{0\}} \frac{(f,v_h)}{\|v_h\|_{H^1}}.
\end{align*}

\begin{theorem}
	\label{thm:mini_aux}
	Let $(u,p)\in V\times Q$ solve \eqref{eq:weak_stokes} with $u \in H^3(\Omega)$ and
	$p\in H^2(\Omega)$, and let
	$(u_h,p_h)\in V_h\times Q_h$ solve \Cref{prob:mini}. Then
	\begin{align}
		\label{eq:mini_aux}
		\| u - u_h\|_{H^1} + \frac{1}{\nu}\|\Lag{1} p - p_h\|_{L^2}
		\lesssim
		\inf_{v_h \in V_h} \|u-v_h\|_{H^1} + \| (\id - \NedI{0}) \Delta u\|_{V_h^*}.
	\end{align}
\end{theorem}
\begin{proof}
	Let $v_h \in V_h$ be arbitrary. Using the triangle inequality we have
	\begin{align} \label{eq:mini_aux_triangle}
		\| u - u_h\|_{H^1} \leq \|u- v_h\|_{H^1} + \| v_h-u_h\|_{H^1}.
	\end{align}
	The first term can be bounded by the infimum over $v_h \in V_h$.
	We continue estimating the second term. Since $\Lag{1} p -p_h \in Q_h$, the inf-sup estimate \eqref{eq:MINI_infsup} gives
	\begin{align}
		\label{eq:mini_infsup_proof}
		\|v_h-u_h\|_{H^1} &+ \frac{1}{\nu}\|\Lag{1} p - p_h\|_{L^2}\\
		&\quad\lesssim \underset{(w_h,q_h)\neq(0,0)}{\sup_{(w_h,q_h)\in V_h\times Q_h}}
		\frac{B_h((v_h-u_h,\Lag{1} p - p_h),(w_h,q_h))}{\nu\|w_h\|_{H^1}+\|q_h\|_{L^2}}.\nonumber
	\end{align}
	Using that $(u_h,p_h)$ solves \Cref{prob:mini} and the strong form
	$-\nu\Delta u+\nabla p=f$, the numerator becomes
	\begin{align*}
		&B_h((v_h-u_h,\Lag{1} p - p_h),(w_h,q_h))\\
		&\quad= \nu(\nabla v_h,\nabla w_h) - (\div(w_h),\Lag{1} p)-(\div(v_h),q_h)
		- (\NedI{0}(-\nu\Delta u+\nabla p),w_h)\\
		&\quad= \nu(\nabla v_h,\nabla w_h)-(\div(v_h),q_h)
		+\nu(\NedI{0}(\Delta u),w_h)
		\underbrace{- (\NedI{0}(\nabla p),w_h)-(\div(w_h),\Lag{1} p)}_{=0},
	\end{align*}
	where the underbraced terms cancel by \Cref{thm:commuting}, integration by parts and using $w_h|_{\partial\Omega}=0$, i.e.
	\begin{align} \label{eq::mini_vanishingpressure}
		-(\NedI{0}(\nabla p),w_h)-(\div(w_h),\Lag{1} p)
		&=-(\nabla \Lag{1} p, w_h)-(\div(w_h),\Lag{1} p)\\
		&=(\Lag{1} p,\div(w_h))-(\div(w_h),\Lag{1} p)=0.\nonumber
	\end{align}
	Adding $-\nu(\Delta u,w_h) - \nu(\nabla u,\nabla w_h)=0$, which vanishes again by integration by parts, and $(\div(u),q_h)=0$, we obtain
	\begin{align*}
		&B_h((v_h - u_h,\Lag{1} p - p_h),(w_h,q_h))\nonumber\\
		&\quad= \nu(\nabla(v_h-u),\nabla w_h) - (\div(v_h-u),q_h)
		+ \nu(\NedI{0}(\Delta u)-\Delta u,w_h)\nonumber\\
		&\quad \leq \nu|v_h- u|_{H^1}|w_h|_{H^1}
		+|v_h - u|_{H^1}\|q_h\|_{L^2}
		+ \nu\|\NedI{0}(\Delta u)-\Delta u\|_{V_h^*}\|w_h\|_{H^1},
	\end{align*}
where we used the Cauchy--Schwarz inequality and
\begin{align} \label{eq:mini_dualnorm_estimate}
		(\NedI{0}(\Delta u)-\Delta u,w_h)&=
		\frac{(\NedI{0}(\Delta u)-\Delta u,w_h)}{\|w_h\|_{H^1}}\|w_h\|_{H^1}\\
		&\leq \|\NedI{0}(\Delta u)-\Delta u\|_{V_h^*}\|w_h\|_{H^1} \nonumber.
\end{align}
	Inserting this bound into \eqref{eq:mini_infsup_proof}, cancelling the common
	denominator gives
	\begin{align*}
		\|v_h-u_h\|_{H^1} &+ \frac{1}{\nu}\|\Lag{1} p - p_h\|_{L^2}
		\lesssim  \|v_h- u\|_{H^1} + \|\NedI{0}(\Delta u)-\Delta u\|_{V_h^*}.
	\end{align*}
	We conclude the proof with \eqref{eq:mini_aux_triangle} and taking the infimum over $v_h \in V_h$.
\end{proof}

\begin{theorem}[$H^1$-error, pressure-robustness]
	\label{thm:mini_h1}
		Let $(u,p)\in V\times Q$ solve \eqref{eq:weak_stokes} with $u \in H^3(\Omega)$ and
	$p\in H^2(\Omega)$, and let
		$(u_h,p_h)\in V_h\times Q_h$ solve \Cref{prob:mini}. Then
	\begin{align*}
		\|u-u_h\|_{H^1} + \frac{1}{\nu}\|\Lag{1} p - p_h\|_{L^2} \lesssim h\,(|u|_{H^2}+|\Delta u|_{H^1}).
	\end{align*}
\end{theorem}
\begin{proof}
	This follows by combining \Cref{thm:mini_aux} with the Lagrange interpolation error estimate
	\eqref{eq:lagrange_error} (with $k=1$, $m = 1$, $j = 2$)
	\begin{align*}
		\inf_{v_h \in V_h} \|u-v_h\|_{H^1} \leq \|u-\Lag{1} u\|_{H^1} \lesssim h\,|u|_{H^2}.
	\end{align*}
	For the consistency error term, we have by the Cauchy--Schwarz inequality and the interpolation error estimate \eqref{eq:nedelec_general_error} (with $k = 0, m = 0, j = 1$)
	\begin{align} \label{eq:hone_mini_laststep}
		\| (\id - \NedI{0}) \Delta u\|_{V_h^*} &= \sup_{v_h\in V_h\setminus\{0\}} \frac{(\Delta u - \NedI{0}(\Delta u), v_h)}{\|v_h\|_{H^1}} \\
		&\le \sup_{v_h\in V_h\setminus\{0\}} \frac{\|\Delta u - \NedI{0}(\Delta u)\|_{L^2} \| v_h\|_{L^2}}{\|v_h\|_{H^1}}
		\lesssim h\,|\Delta u|_{H^1}. \nonumber
	\end{align}

\end{proof}

Notably, in the proof of \Cref{thm:mini_h1} we only used the cruder $L^2$-bound \eqref{eq:nedelec_general_error}, see \Cref{rem:h1_consistency}, for the consistency term, which requires less regularity compared to the error estimate in \Cref{thm:mini_l2} below. The latter will rely on the superconvergent estimate from	\Cref{prop:consistency_high}.

\begin{remark} \label{rem:mini_pressureregularity}
We want to emphasize that, even though the bounds in \Cref{thm:mini_aux} and \Cref{thm:mini_h1} do not explicitly depend on the continuous pressure $p$, we still needed the regularity assumption $p\in H^2(\Omega)$ in order to ensure that the Lagrange interpolation $\Lag{1} p$ is well-defined. In the following we  provide an error estimate for the pressure where this increased regularity assumption is more obvious.
\end{remark}

\begin{theorem}[Pressure error]
	\label{thm:mini_pressure}
	Let $u\in H^3(\Omega)$, $p\in H^2(\Omega)$ solve \eqref{eq:weak_stokes} and let
	$p_h\in Q_h$ solve \Cref{prob:mini}. Then there holds
the error estimate
	\begin{align*}
		\|p-p_h\|_{L^2} \lesssim h\,(\nu|u|_{H^2}+\nu|\Delta u|_{H^1}+ h |p|_{H^2}).
	\end{align*}
\end{theorem}
\begin{proof}

	 Using the triangle inequality we have
	\begin{align*}
		\| p - p_h\|_{L^2} \leq \|p - \Lag{1} p\|_{L^2} + \|\Lag{1} p - p_h\|_{L^2}.
	\end{align*}
	The first term can be bounded by the Lagrange interpolation error estimate \eqref{eq:lagrange_error} (with $k=1$, $m = 0$, $j = 2$). The second term can be bounded by \Cref{thm:mini_h1}.
\end{proof}

\begin{remark}[Superconvergence for small $\nu$]
	\label{rem:mini_superconvergence}
	\Cref{thm:mini_pressure} explains why, for small $\nu$ (where the pressure error dominates), one may observe an experimental
	convergence rate of order $2$ instead of $1$ for the pressure -- see \Cref{sec:numerics}.
	\eremk
\end{remark}

\begin{theorem}[$L^2$-error, pressure-robust]
	\label{thm:mini_l2}
	Let $u\in H^4(\Omega)$ be the velocity and $p\in H^2(\Omega)$ the pressure solving
	\eqref{eq:weak_stokes}, and let
	$u_h\in V_h$ solve \Cref{prob:mini}.
	Additionally, suppose the dual solution $(w,r)$ of
	\eqref{eq:dual_stokes} is $H^2$/$H^1$-regular, i.e.
	$\nu\|w\|_{H^2}+\|r\|_{H^1}\leq\|e\|_{L^2}$. Then
	\begin{align*}
		\|u-u_h\|_{L^2} \lesssim h^2\,(|u|_{H^2}+|\Delta u|_{H^1}+\|\Delta u\|_{H^2}).
	\end{align*}
\end{theorem}
\begin{proof}
	Set $e\coloneqq u-u_h$ and test \eqref{eq:dual_stokes} with $v=e$, so that
	$\|e\|_{L^2}^2 = \nu\,a(u-u_h,w)+b(u-u_h,r)$. Subtracting the (vanishing) discrete
	consistency relation
	$\nu\,a(u-u_h,\Cle{1} w) + b(\Cle{1} w,p-p_h)
	+b(u-u_h,\Cle{1} r) - (f-\NedI{0} f,\Cle{1} w) = 0$,
	and using $f=-\nu\Delta u+\nabla p$ together with \Cref{thm:commuting} yields
	\begin{align*}
		\|e\|_{L^2}^2 &=
		\nu\,a(u-u_h,w-\Cle{1} w) + b(u-u_h,r-\Cle{1} r)
		+(\div(\Cle{1} w),\Lag{1} p - p_h)\\
		&\quad -\nu(\Delta u - \NedI{0}(\Delta u),\Cle{1} w).
	\end{align*}
	The first two terms are bounded by Cauchy--Schwarz, \Cref{thm:mini_h1} and
	\eqref{eq:quasi_interpolation} by $h^2(|u|_{H^2}+|\Delta u|_{H^1})\big(\nu|w|_{H^2}+|r|_{H^1}\big)$;
	for the third, inserting $(\div(w),\cdot)=0$, Cauchy--Schwarz,
	\eqref{eq:quasi_interpolation} and again \Cref{thm:mini_h1} give
	$h^2(|u|_{H^2}+|\Delta u|_{H^1})\,\nu|w|_{H^2}$; for the fourth,
	\Cref{prop:consistency_low} (with $f=\Delta u$, $g=\Cle{1} w$, which
	is bounded in $H^1(\Omega)$ by $\|w\|_{H^1}$, see \eqref{eq:quasi_interpolation_stability}) gives $h^2\|\Delta
	u\|_{H^2}\,\nu\|w\|_{H^1}$. Collecting all four terms, dividing by
	$\|e\|_{L^2}$ and using $H^2$/$H^1$-regularity of $(w,r)$ concludes the proof.
\end{proof}

\section{The stabilized $P^kP^k$ element}
\label{sec:stabpkpk}
For equal-order velocity/pressure pairs the choice $V_h\coloneqq [\Lag{k}(\mathcal{T}_h)]^2 \cap V$ and
	$Q_h\coloneqq \Lag{k}(\mathcal{T}_h)\cap Q$, $k\geq1$, violates the discrete LBB-condition and needs to be stabilized
\cite{franca_hughes_stenberg_1993}. Since velocity and pressure are approximated by
polynomials of the same degree, the commuting diagram (\Cref{thm:commuting}) now
prescribes the Nédélec operator of order $k-1$,
i.e. $\NedI{k-1}$, which commutes with the Lagrange operator
$\Lag{k}$ of the pressure space, so that no order is
lost. We define the (modified) residual-based stabilizing bilinear
form and right hand side
\begin{subequations}\label{eq:stab_forms}
\begin{align}
	c((u_h,p_h),(v_h,q_h)) &\coloneqq -\frac{\alpha h^2}{\nu}\sum_T \big(-\nu\Delta u_h+\nabla p_h,\, -\nu\Delta v_h+\nabla q_h\big)_{L^2(T)},\\
	g(v_h,q_h) &\coloneqq -\frac{\alpha h^2}{\nu}\sum_T\big(\NedI{k-1} f,\, -\nu\Delta v_h+\nabla q_h\big)_{L^2(T)},
\end{align}
\end{subequations}
for a stabilization parameter $\alpha>0$ chosen appropriately. With
\begin{align*}
	B_h((u_h,p_h),(v_h,q_h)) \coloneqq \nu\,a(u_h,v_h)+b(v_h,p_h)+b(u_h,q_h)+c((u_h,p_h),(v_h,q_h)),
\end{align*}
the discrete problem reads:

\begin{problem}
	\label{prob:stabpkpk}
	Find $(u_h,p_h)\in V_h\times Q_h$
	such that
	\begin{alignat*}{2}
		B_h((u_h,p_h),(v_h,q_h)) = (\NedI{k-1} f,v_h)+g(v_h,q_h) && \quad \forall (v_h,q_h)\in V_h\times Q_h.
	\end{alignat*}
\end{problem}

Only $f$ is interpolated -- the right hand side $g$ of the incompressibility constraint is
chosen to contain the very same interpolated force, so that no \emph{additional}
consistency error is introduced there and is crucial in order to maintain pressure-robustness. Stability of the un-interpolated method is classical
\cite{franca_hughes_stenberg_1993} and gives for sufficient $\alpha>0$ the inf-sup estimate
\begin{align}
	\label{eq:stabpkpk_infsup}
	\|u_h\|_{H^1}+\frac{1}{\nu}\|p_h\|_{L^2} &\lesssim
	\underset{(v_h,q_h)\neq(0,0)}{\sup_{(v_h,q_h)\in V_h\times Q_h}} \frac{B_h((u_h,p_h),(v_h,q_h))}{\nu\|v_h\|_{H^1}+\|q_h\|_{L^2}}
	\quad \forall (u_h,p_h)\in V_h\times Q_h.
\end{align}

As for the MINI element, the analysis rests on a single quasi-best-approximation
estimate, in the same spirit as \Cref{thm:mini_aux}. However, the stabilization term forces us to
treat two additional error contributions differently. For the first term we will particularly need a
\emph{local, broken} version of $\|\cdot\|_{V_h^*}$. For $f\in L^2(\Omega)$
we define
\begin{align}
	\label{def:broken_dual_norm}
	\|f\|_{V_h^*(T)} \coloneqq \sup_{\psi_h\in [P^k(T)]^2\setminus\{0\}}
	\frac{(f,\psi_h)_{L^2(T)}}{\|\psi_h\|_{H^1(T)}},
	\quad
	\|f\|_{V_h^*(\mathcal{T}_h)}\coloneqq\Big(\sum_T\|f\|_{V_h^*(T)}^2\Big)^{1/2}.
\end{align}
Note that $\|f\|_{V_h^*}\le\|f\|_{V_h^*(\mathcal{T}_h)}\le\|f\|_{L^2}$. The second contribution that appears is a term that is known from standard error estimates for stabilized methods and will be bounded as usual, see e.g.\ \cite{MR3425376}. Note that in \cite{MR3425376} the authors improved this local consistency error by means of \textit{Gudi's trick}, \cite{MR2684360}, by adding oscillations of the force term. Since it is not inherently clear how this can be extended to a pressure robust setting, we do not apply this technique here.

\begin{theorem}
	\label{thm:stabpkpk_bestapprox}
	Let $(u,p)\in V\times Q$ solve \eqref{eq:weak_stokes} with $u\in H^{3}(\Omega)$ and
	$p\in H^2(\Omega)$, and let
	$(u_h,p_h)\in V_h\times Q_h$ solve \Cref{prob:stabpkpk}. Then
	\begin{align}
		\label{eq:stabpkpk_bestapprox}
		\|u - u_h\|_{H^1}+\frac{1}{\nu}\|\Lag{k} p - p_h\|_{L^2}
		&\lesssim \inf_{v_h\in V_h} \Big(\|u-v_h\|_{H^1} + h\,|u-v_h|_{H^2(\mathcal{T}_h)} \Big)\\
		&\qquad \qquad+ \|\Delta u-\NedI{k-1}(\Delta u)\|_{V_h^*(\mathcal{T}_h)}. \nonumber
	\end{align}
\end{theorem}
\begin{proof} The proof follows along the same lines as that of \Cref{thm:mini_aux}. The crucial difference is the handling of the stabilization term and the additional consistency term due to $g$, which requires a slightly more careful treatment.

	Let $v_h \in V_h$ be arbitrary. Using the triangle inequality we have
	\begin{align}
		\label{eq:stabpkpk_bestapprox_triangle}
		\|u-u_h\|_{H^1} \leq \|u-v_h\|_{H^1} + \|v_h-u_h\|_{H^1}.
	\end{align}
	The first term is bounded by the infimum over $v_h\in V_h$. We continue estimating the
	second term. Since $\Lag{k} p - p_h \in Q_h$, the inf-sup estimate \eqref{eq:stabpkpk_infsup}
	applied to $(v_h-u_h,\Lag{k} p - p_h)\in V_h\times Q_h$ gives
	\begin{align}
		\label{eq:stabpkpk_infsup_proof}
		\|v_h-u_h\|_{H^1} &+ \frac{1}{\nu}\|\Lag{k} p - p_h\|_{L^2}
		\lesssim \underset{(w_h,q_h)\neq(0,0)}{\sup_{(w_h,q_h)\in V_h\times Q_h}}
		\frac{B_h((v_h-u_h,\Lag{k} p - p_h),(w_h,q_h))}{\nu\|w_h\|_{H^1}+\|q_h\|_{L^2}}.
	\end{align}
	Using linearity of $B_h$, that $(u_h,p_h)$ solves \Cref{prob:stabpkpk}, and
	$f=-\nu\Delta u+\nabla p$, the numerator becomes
	\begin{align*}
		&B_h((v_h - u_h,\Lag{k} p - p_h),(w_h,q_h))\\
		&\quad= B_h((v_h,\Lag{k} p),(w_h,q_h)) - (\NedI{k-1} f, w_h) - g(w_h,q_h)\\
		&\quad= \nu\,a(v_h,w_h)+b(w_h,\Lag{k} p)+b(v_h,q_h)
		+\nu(\NedI{k-1}(\Delta u),w_h) - (\NedI{k-1}(\nabla p),w_h)\\
		&\qquad +\frac{\alpha h^2}{\nu}\sum_{T}\Big(\NedI{k-1}(-\nu\Delta u+\nabla p) - (-\nu\Delta v_h+\nabla \Lag{k} p),\,-\nu\Delta w_h+\nabla q_h\Big)_{L^2(T)}.
	\end{align*}
	By \Cref{thm:commuting}, $-(\NedI{k-1}(\nabla p),w_h)
	=(\Lag{k} p,\div(w_h)) = -b(w_h,\Lag{k} p)$,
	which cancels the term $b(w_h,\Lag{k} p)$ already present; likewise, in
	the stabilization sum, $\NedI{k-1}(\nabla p) = \nabla
	\Lag{k} p$, cancelling all pressure terms inside the sum.
	With $\phi_h \coloneqq -\nu\Delta w_h+\nabla q_h$, what remains is
	\begin{align*}
		B_h((v_h - u_h,\Lag{k} p - p_h),(w_h,q_h))
		=& \nu\,a(v_h,w_h) + b(v_h,q_h)
		+\nu(\NedI{k-1}(\Delta u),w_h)\\
		& +\alpha h^2 \sum_T \big(-\NedI{k-1}(\Delta u)+\Delta v_h,\,\phi_h\big)_{L^2(T)}.
	\end{align*}
	Adding $-\nu(\Delta u,w_h)-\nu(\nabla u,\nabla w_h)=0$ (which vanishes due to integration by parts) and $-
	b(u,q_h)=0$, and $\pm\Delta u$ inside the stabilization sum,
	\begin{align*}
		&B_h((v_h - u_h,\Lag{k} p - p_h),(w_h,q_h))\\
		&\quad= \nu\,a(v_h - u,w_h) + b(v_h - u,q_h)
		+\nu(\NedI{k-1}(\Delta u)-\Delta u,w_h)\\
		&\qquad +\alpha h^2\sum_T \big(\Delta u - \NedI{k-1}(\Delta u)+\Delta v_h-\Delta u,\,\phi_h\big)_{L^2(T)}\\
		&\quad\overset{\text{C.S. and } \eqref{eq:mini_dualnorm_estimate}}{\leq}
		\nu|v_h - u|_{H^1}|w_h|_{H^1} + |v_h - u|_{H^1}\|q_h\|_{L^2}
		+\nu\|(\id-\NedI{k-1})\Delta u\|_{V_h^*}\|w_h\|_{H^1}\\
		&\qquad + \alpha h^2\sum_T\big(\Delta u - \NedI{k-1}(\Delta u),\,\phi_h\big)_{L^2(T)}
		+ \alpha h^2\sum_T \|\Delta v_h-\Delta u\|_{L^2(T)}\,\|\phi_h\|_{L^2(T)}.
	\end{align*}
	Since $\phi_h$ is piecewise polynomial, a scaling argument gives
	\begin{align} \label{eq:stabpkpk_phi_scaling}
		\|\phi_h\|_{L^2(T)} = \|-\nu\Delta w_h + \nabla q_h\|_{L^2(T)} &\lesssim \nu h^{-1}\|w_h\|_{H^1(T)}+h^{-1}\|q_h\|_{L^2(T)}.
	\end{align}

	For the first stabilization-sum term we can not proceed as for the third term, i.e. using \eqref{eq:mini_dualnorm_estimate} to pair $\Delta u - \NedI{k-1}(\Delta u)$ against $\phi_h$ through the dual norm $\|\cdot\|_{V_h^*}$. This is due to $\phi_h$ being only piecewise polynomial, but generally discontinuous across element interfaces. We instead pair it through the local version of the dual norm, see \Cref{def:broken_dual_norm}.
	Together with the inverse inequality $\|\phi_h\|_{H^1(T)}\lesssim
	h^{-1}\|\phi_h\|_{L^2(T)}$ we have
	\begin{align*}
		\big(\Delta u - \NedI{k-1}(\Delta u),\phi_h\big)_{L^2(T)}
		&\leq \|\Delta u - \NedI{k-1}(\Delta u)\|_{V_h^*(T)}\,\|\phi_h\|_{H^1(T)}\\
		&\lesssim \|\Delta u - \NedI{k-1}(\Delta u)\|_{V_h^*(T)}\,h^{-1}\|\phi_h\|_{L^2(T)}.
	\end{align*}
	Combined with the scaling \eqref{eq:stabpkpk_phi_scaling} and Cauchy--Schwarz in the sum over $T$ this gives
	\begin{align*}
		\alpha h^2\sum_T\big(\Delta u-\NedI{k-1}&(\Delta u),\phi_h\big)_{L^2(T)}
		\lesssim \|\Delta u-\NedI{k-1}(\Delta u)\|_{V_h^*(\mathcal{T}_h)}\big(\nu\|w_h\|_{H^1}+\|q_h\|_{L^2}\big).
	\end{align*}

	For the second stabilization-sum term we use Cauchy--Schwarz and the scaling \eqref{eq:stabpkpk_phi_scaling} to get
	\begin{align} \label{eq::stabpkpk_stab_sum_scaling}
		\alpha h^2\sum_T\|\Delta v_h-\Delta u\|_{L^2(T)}\|\phi_h\|_{L^2(T)}
		\lesssim  h\,|v_h-u|_{H^2(\mathcal{T}_h)}\big(\nu\|w_h\|_{H^1}+\|q_h\|_{L^2}\big).
	\end{align}
	Inserting these bounds into
	\eqref{eq:stabpkpk_infsup_proof} and cancelling the common denominator gives
	\begin{align*}
		\|v_h-u_h\|_{H^1} + \frac{1}{\nu}\|\Lag{k} p - p_h\|_{L^2}
		&\lesssim  \|v_h- u\|_{H^1} + h\,|v_h-u|_{H^2(\mathcal{T}_h)}\\
		&\quad  + \|\Delta u - \NedI{k-1}(\Delta u)\|_{V_h^*(\mathcal{T}_h)},
	\end{align*}
	where we have used that $\|(\id-\NedI{k-1})\Delta u\|_{V_h^*} \le \|\Delta u - \NedI{k-1}(\Delta u)\|_{V_h^*(\mathcal{T}_h)}$.
	We conclude the proof with \eqref{eq:stabpkpk_bestapprox_triangle} and taking the infimum
	over $v_h \in V_h$.
\end{proof}

\begin{remark}[The $k=1$ case]
	\label{rem:p1p1}
	For the special case $k = 1$, we have $\Delta v_h = \Delta w_h = 0$ for all $v_h,w_h \in V_h$, so that the stabilization term reduces to a pressure-Laplacian stabilization. In this case, the same proof as above holds, but we want to emphasize, that in \eqref{eq::stabpkpk_stab_sum_scaling} the $v_h$ vanishes. However, since for $k=1$ we only expect a linear convergence, the remaining term $h | u |_{H^2}$ is already of optimal order.
	\eremk
\end{remark}

\begin{theorem}[$H^1$-error, pressure-robustness]
	\label{thm:stabpkpk_aux}
	Let $(u,p)\in V\times Q$ solve \eqref{eq:weak_stokes} with $u\in H^{s}(\Omega)$, with $s \ge 3$, and
	$p\in H^2(\Omega)$, and let
	$(u_h,p_h)\in V_h\times Q_h$ solve \Cref{prob:stabpkpk}. Then
	\begin{align}
		\label{eq:stabpkpk_aux}
		\|u - u_h\|_{H^1}+\frac{1}{\nu}\|\Lag{k} p - p_h\|_{L^2}
\lesssim h^l\,(|u|_{H^{l+1}}+|\Delta u|_{H^l}),
	\end{align}
	with $l=\min\{k,s-2\}$.
\end{theorem}
\begin{proof}
	Bounding the infimum in \Cref{thm:stabpkpk_bestapprox} by the choice $v_h = \Lag{l} u$ gives by the Lagrange interpolation error estimate \eqref{eq:lagrange_error} (with $m=1,j=l+1$),
	\begin{alignat*}{2}
		\|u-\Lag{l}u\|_{H^1}&\lesssim h^l|u|_{H^{l+1}}   && \quad \text{using \eqref{eq:lagrange_error} with $m=1,j=l+1$},\\
		|u-\Lag{l}u|_{H^2}&\lesssim h^{l-1}|u|_{H^{l+1}} && \quad \text{using \eqref{eq:lagrange_error} with $m=2,j=l+1$},
	\end{alignat*}
	so that $h\,|u-\Lag{l}u|_{H^2(\mathcal{T}_h)}	\lesssim h^l|u|_{H^{l+1}}$ as well. As discussed in \Cref{rem:h1_consistency}, using \eqref{eq:nedelec_general_error} (with $m=0, j=l$) and $\|\cdot\|_{V_h^*(\mathcal{T}_h)}\le\|\cdot\|_{L^2}$ we also have
	\begin{align*}
		\|\Delta u - \NedI{k-1}(\Delta u)\|_{V_h^*(\mathcal{T}_h)}\lesssim
	h^l|\Delta u|_{H^l}.
	\end{align*}
	 Collecting all four terms in \eqref{eq:stabpkpk_bestapprox} gives
	\eqref{eq:stabpkpk_aux}.
\end{proof}

It is worth noting that the estimate of \Cref{thm:stabpkpk_bestapprox} only assumed $p\in H^2(\Omega)$ while a higher regularity of the velocity is required in order to get the desired high-order convergence bound. The pressure regularity is solely needed for the Lagrange interpolation $\Lag{k} p$ to be well defined, and in accordance with \Cref{rem:mini_pressureregularity}. However, since the pressure-robust velocity error does not depend on the pressure, no additional regularity of the pressure is needed (in contrast to the following pressure error estimate).

\begin{theorem}[Pressure error]
	\label{thm:stabpkpk_pressure}
		Let $(u,p)\in H^{s}(\Omega)\times H^{s-1}(\Omega)$, with $s \ge 3$, solve \eqref{eq:weak_stokes}
	and $(u_h,p_h)$ solve \Cref{prob:stabpkpk}. Then
	\begin{align*}
		\|p-p_h\|_{L^2} \lesssim h^l\,(\nu|u|_{H^{l+1}}+\nu|\Delta u|_{H^l}+h |p|_{H^{l+1}}),
	\end{align*}
	with $l=\min\{k,s-2\}$.
\end{theorem}
\begin{proof}
	The proof follows with the same steps as for \Cref{thm:mini_pressure}, using the Lagrange interpolation error estimates \eqref{eq:lagrange_error} and \Cref{thm:stabpkpk_aux}.
\end{proof}

\begin{remark}
	In \Cref{thm:stabpkpk_pressure} one could replace the term $h |p|_{H^{l+1}}$ by $|p|_{H^{l}}$ since the remaining terms have a prefactor $h^l$ anyway. However - similarly as for the MINI elements - the present estimate shows the pre-asymptotic higher convergence error of the pressure when the viscosity is small, see \Cref{rem:mini_superconvergence}.
	\eremk
\end{remark}

We conclude this section with the $L^2$-norm error estimate, which is where
the super-convergent, pressure-independent consistency error of \Cref{prop:consistency_high}
enters and is indispensable for retaining optimal order.

\begin{theorem}[$L^2$-error, pressure-robust, optimal order]
	\label{thm:stabpkpk_l2}
		Let $(u,p)\in V\times Q$ solve \eqref{eq:weak_stokes} with $u\in H^{s}(\Omega)$, with $s \ge 4$, and
	$p\in H^2(\Omega)$, and assume that the dual solution $(w,r)$ of \eqref{eq:dual_stokes} is
	$H^2$/$H^1$-regular. Then
	\begin{align*}
		\|u-u_h\|_{L^2} \lesssim h^{l+1}\,(|u|_{H^{l+1}}+\|\Delta u\|_{H^{l+1}}),
	\end{align*}
	with $l=\min\{k,s-3\}$.
\end{theorem}
\begin{proof}
	We test \eqref{eq:dual_stokes} with $v=e\coloneqq u-u_h$ as in \Cref{thm:mini_l2} and
	subtract the (vanishing) discrete relation obtained from \Cref{prob:stabpkpk} tested with
	$v_h=\Cle{1} w$, $q_h=\Cle{1} r$. Using
	$f=-\nu\Delta u+\nabla p$, and \Cref{thm:commuting} exactly as
	in the proof of \Cref{thm:mini_l2} (now also exploiting \Cref{thm:commuting} inside the
	stabilization sum), one obtains
	\begin{align}
		\|e\|_{L^2}^2 &=
		\nu\,a(u-u_h,w-\Cle{1} w) + b(u-u_h,r-\Cle{1} r)
		+ (\div(\Cle{1} w), p_h - \Lag{k} p) \nonumber\\
		&\quad -\nu(\Delta u - \NedI{k-1}(\Delta u), \Cle{1} w) \nonumber\\
		&\quad +\alpha h^2\sum_T (\Delta u_h - \Delta u, \nabla \Cle{1} r)_{L^2(T)} \label{eq:stab_termA}\\
		&\quad +\alpha h^2\sum_T (\Delta u - \NedI{k-1}(\Delta u), \nabla \Cle{1} r)_{L^2(T)} \label{eq:stab_termB}\\
		&\quad -\frac{\alpha h^2}{\nu}\sum_T (\nabla(p_h-\Lag{k} p), \nabla \Cle{1} r)_{L^2(T)}, \label{eq:stab_termC}
	\end{align}
	where we used that $\Delta \Cle{1} w = 0$.
	For the first four terms we proceed exactly as in the proof of \Cref{thm:mini_l2} but exploiting the error estimate \Cref{thm:stabpkpk_aux} and
	\Cref{prop:consistency_high} (for the last term)
	to bound \begin{align*}
		&\nu\,a(u-u_h,w-\Cle{1} w) + b(u-u_h,r-\Cle{1} r) \\
		&+ (\div(\Cle{1} w), p_h - \Lag{k} p)
		 -\nu(\Delta u - \NedI{k-1}(\Delta u), \Cle{1} w)\\
		&\lesssim h^{l+1}(|u|_{H^{l+1}}+\|\Delta u\|_{H^{l+1}})
		(\nu|w|_{H^2}+|r|_{H^1}).
	\end{align*}
	Note, that the estimate of the last term via the super-convergence result from \Cref{prop:consistency_high} is the one, where the higher regularity of $\Delta u$  was needed.

	Similarly as in the proof of \Cref{thm:stabpkpk_bestapprox}, the stabilization terms need to be bound carefully. We start with the first sum \eqref{eq:stab_termA}, adding and subtracting $\Delta \Lag{k} u$ and using Cauchy--Schwarz gives
	\begin{align*}
		\alpha h^2\sum_T (\Delta u_h - \Delta u, \nabla \Cle{1} r)_{L^2(T)}
		\lesssim
		&\alpha h^2\sum_T \|\Delta (u_h - \Lag{k} u)\|_{L^2(T)}\|\nabla \Cle{1} r\|_{L^2(T)}\\
		+ &\alpha h^2\sum_T \|\Delta (\Lag{k} u - u)\|_{L^2(T)}\|\nabla \Cle{1} r\|_{L^2(T)}.
	\end{align*}
	Since $u_h - \Lag{k} u$ is a polynomial, we use an inverse inequality and the $H^1$-stability of $\Cle{1}$, see \eqref{eq:quasi_interpolation_stability}, to get
	\begin{align*}
		\alpha h^2\sum_T \|\Delta (u_h - \Lag{k} u)\|_{L^2(T)}\|\nabla \Cle{1} r\|_{L^2(T)}
		&\lesssim  h \sum_T \|\nabla (u_h - \Lag{k} u)\|_{L^2(T)}\|\nabla \Cle{1} r\|_{L^2(T)}\\
		&\lesssim  h \|u_h - \Lag{k} u\|_{H^1}\|r\|_{H^1}.
	\end{align*}
	Adding and subtracting $u$ and using the triangle inequality, we can bound this term by \Cref{thm:stabpkpk_aux}  and the Lagrange interpolation error estimate \eqref{eq:lagrange_error} (with $m=1, j = l+1$) to get
	\begin{align*}
		 h \|u_h - \Lag{k} u\|_{H^1}\|r\|_{H^1}
		 &\le
h (\|u_h - u\|_{H^1}+\|u - \Lag{k} u\|_{H^1}) \|r\|_{H^1} \\
		&\lesssim  h^{l+1}(|u|_{H^{l+1}}+\|\Delta u\|_{H^{l}})\|r\|_{H^1}.
	\end{align*}
	For the second term, we use the Lagrange interpolation error estimate \eqref{eq:lagrange_error} (with $m=2, j = l+1$) to similarly get
	\begin{align*}
		\alpha h^2\sum_T \|\Delta (\Lag{k} u - u)\|_{L^2(T)}\|\nabla \Cle{1} r\|_{L^2(T)}
		&\lesssim h^2 | \Lag{k} u - u|_{H^2(\mathcal{T}_h)}\|r\|_{H^1}\\
		&\lesssim  h^{l+1}(|u|_{H^{l+1}})\|r\|_{H^1}.
	\end{align*}

	For the second stabilization-sum term \eqref{eq:stab_termB}, we proceed similarly as for the later term, using Cauchy--Schwarz and the approximation properties of the Nédélec interpolation operator \eqref{eq:nedelec_general_error} (with $m=0$, $j = l$), to get
	\begin{align*}
		 \alpha h^2\sum_T (\Delta u - \NedI{k-1}(\Delta u), \nabla \Cle{1} r)_{L^2(T)}
		&\lesssim h^2 \sum_T \|\Delta u - \NedI{k-1}(\Delta u)\|_{L^2(T)}\|\nabla \Cle{1} r\|_{L^2(T)}\\
		&\lesssim h^2 \|\Delta u - \NedI{k-1}(\Delta u)\|_{L^2}\|r\|_{H^1}\\
		&\lesssim h^{l+2}\|\Delta u\|_{H^{l+1}}\|r\|_{H^1}
		\lesssim h^{l+1}\|\Delta u\|_{H^{l+1}}\|r\|_{H^1},
	\end{align*}
	where the last step drops one power of $h$; the additional power of $h$ was inherited from the choice $j=l$. Although one could use $j = l-1$ here, this would be an issue in the borderline case $l=1$, since the (standard) Nédélec interpolation operator is not $L^2$-stable, cf.\ \eqref{eq:nedelec_general_error}. As we already needed the higher regularity of $\Delta u$ at the beginning of the proof (due to \Cref{prop:consistency_high}) we opted to use $j=l$ here as well.

	For the last stabilization-sum term \eqref{eq:stab_termC}, we proceed similarly. Since $p_h - \Lag{k} p$ is a polynomial, we can use an inverse inequality and the $H^1$-stability of $\Cle{1}$ to get
	\begin{align*}
		\frac{\alpha h^2}{\nu}\sum_T (\nabla(p_h-\Lag{k} p), \nabla \Cle{1} r)_{L^2(T)}
		&\lesssim \frac{\alpha h}{\nu}\sum_T \|p_h - \Lag{k} p\|_{L^2(T)}\|\nabla \Cle{1} r\|_{L^2(T)}\\
		&\lesssim \frac{\alpha h}{\nu}\|p_h - \Lag{k} p\|_{L^2}\|r\|_{H^1},\\
		&\lesssim h^{l+1}(|u|_{H^{l+1}}+\|\Delta u\|_{H^{l}})\|r\|_{H^1},
	\end{align*}
	where we used \Cref{thm:stabpkpk_aux} in the last step.

	Collecting all terms, dividing by
	$\|e\|_{L^2}$ and using $H^2$/$H^1$-regularity of $(w,r)$ concludes the proof.
\end{proof}

\section{The Taylor--Hood element: a cautionary example}
\label{sec:taylorhood}
The commonly known family of Taylor--Hood elements \cite{TAYLOR197373}, $V_h\coloneqq [\Lag{k+1}(\mathcal{T}_h)]^2\cap V$, $Q_h\coloneqq \Lag{k}(\mathcal{T}_h)\cap Q$, is
inf-sup stable without any stabilization \cite{GiraultRaviart86}. Since velocity and
pressure are now approximated by polynomials of \emph{different} degree, the commuting
diagram forces the choice $\NedI{k-1}$ -- the operator
that commutes with the \emph{pressure} space $\Lag{k}$, one degree
below the velocity space $\Lag{k+1}$. Repeating the arguments of
\Cref{sec:mini,sec:stabpkpk} yields, for the lowest order case
$k=1$ (the classical $P^2P^1$ element) and $(u,p) \in H^4(\Omega) \times H^2(\Omega)$ the pressure robust and optimal error estimates
\begin{align*}
	\|u-u_h\|_{H^1} &\lesssim h^2 (|u|_{H^3}+\|\Delta u\|_{H^2}), \\
	\|p-p_h\|_{L^2} &\lesssim h^2 (\nu |u|_{H^3}+\nu \|\Delta u\|_{H^2}+|p|_{H^2}).
\end{align*}
where the interpolation operator $\NedI{0}$ was used in the right hand side - similar to the MINI element. To derive these estimates, one follows the same steps as in the proof of \Cref{thm:mini_h1} but uses the super-convergent consistency error estimate of \Cref{prop:consistency_low} in the last step \eqref{eq:hone_mini_laststep} to bound $\| (\id - \NedI{0}) \Delta u\|_{V_h^*}$.

In the $L^2$ norm, however, we only get the bound
\begin{align*}
	\|u-u_h\|_{L^2} \lesssim h^2\,(h|u|_{H^3}+\|\Delta u\|_{H^2}),
\end{align*}
which is only \emph{quadratic}, one order below the cubic rate of the standard,
non-pressure-robust $P^2P^1$ method: the term $\nu(\Delta u -
\NedI{0}(\Delta u), \Cle{1} w)$ - following the same steps as in the proof of \Cref{thm:mini_l2} - is now only
of order $h^2$  by \Cref{prop:consistency_low}. Since
$\NedI{k-1}$ is tied to the pressure space by the
commuting diagram, no choice of $k$ can improve on this loss of one order for the
Taylor--Hood family. So far we have not been able to find a remedy for this, and we leave it as an open problem.

\section{Towards relaxed regularity assumptions}
\label{sec:discussion_regularity}
Throughout the analysis we imposed a comparatively high regularity on the exact solution
$(u,p)$ and, correspondingly, on the force $f$ -- e.g.\ $u\in H^{k+2}(\Omega)$,
$p\in H^2(\Omega)$ already for the $H^1$-error alone, and more for the $L^2$-error via
\Cref{prop:consistency_high}. We want to comment on why this is needed and how it could be
avoided.

This requirement has two, closely related, origins. First, the force is interpolated by the
\emph{canonical} (degree-of-freedom based) Nédélec operator $\NedI{k-1}$, whose edge and face
moments are only classically well-defined for sufficiently regular input -- already the
lowest-order case $\NedI{0} v \coloneqq \sum_E\big(\int_E v\cdot\tau_E\,\ds\big)\varphi_E$
requires a well-defined tangential trace on every edge, i.e.\ $v\in H^1(\Omega)$. Since
$f=-\nu\Delta u+\nabla p$, this in particular means $\Delta u$ and $\nabla p$ must be regular
enough for $\NedI{k-1}$ to make sense. Second, and via exactly the commuting diagram that drives the whole method
(\Cref{thm:commuting}), the same obstruction reappears one level up for the pressure: the
canonical Lagrange operator $\Lag{k}$ has nodal (point-evaluation) degrees of freedom already
at the vertices, for every $k\ge1$, and these are only classically well-defined once
$p\in H^2(\Omega)$. The higher regularity of $u$ itself then enters independently, simply because the best-approximation and optimal-order estimates measure the error against $\Lag{k}u$ and $\NedI{k-1}(\Delta u)$ in progressively higher-order Sobolev norms.

A natural remedy is to replace the canonical interpolation operators by \emph{commuting
quasi-interpolation} operators in the spirit of Clément \cite{Clement1975} and Scott--Zhang
\cite{ScottZhang1990}: i.e. operators defined via local $L^2$-averaging over vertex or edge
patches, which are consequently well-defined -- and stable -- already for merely
$L^2(\Omega)$ or locally $H^1$ data. The difficulty is that a quasi-interpolation of $f$ has to be
paired with a \emph{commuting} quasi-interpolation of the pressure.

Such commuting quasi-interpolation operators, spanning the discrete de Rham complex rather
than a single space, have indeed been constructed and studied in the literature
\cite{Schoberl2008, MR2373181, MR3246803}, and very recently \cite{ErnGuzmanPotuVohralik2025}. So far, however, these
constructions have been used almost exclusively as \emph{analysis} tools -- for a posteriori
error estimation, multigrid theory, and the abstract theory of finite element exterior
calculus -- rather than being incorporated into an actual implementation of a method like the
one proposed here. For concreteness and ease of implementation, we have therefore chosen to
rely on the standard, canonical interpolation operators throughout; replacing them by a
commuting quasi-interpolation pair, thereby relaxing the regularity assumptions on $f$ down to $L^2(\Omega)$ (respectively locally $H^1$) data, is a natural and, we believe,
promising direction for future work.

\section{Numerical examples}
\label{sec:numerics}
All experiments were implemented in Python using the finite element software
Netgen/NGSolve \cite{netgen,ngsolve}.\footnote{The code and raw convergence data underlying the numerical experiments in this section are available at \url{https://doi.org/10.25592/uhhfdm.21546}.}

We choose the unit square $\Omega\coloneqq(0,1)^2$ and let the exact (smooth) velocity and pressure be given by
$u(x,y)\coloneqq\curl\big(x^2(1-x)^2y^2(1-y)^2\big)\in [H^1_0(\Omega)]^2$ and $p(x,y)\coloneqq x^5+y^5-1/3\in L^2_0(\Omega) $.
In the following convergence results we consider a structured mesh of $230$ triangles which is subsequently uniformly refined.
We discuss the convergence of the velocity error $\|u-u_h\|_{L^2}$ and $\|\nabla(u-u_h)\|_{L^2}$, as well as the pressure error $\|p-p_h\|_{L^2}$ together with
the experimental order of convergence (eoc) for two viscosities, $\nu=1$ and $\nu=10^{-6}$. In the case of stabilized methods we choose the stabilization parameter $\alpha=10^{-3}$. The discrete solutions are either the solutions of the pressure-robust methods \Cref{prob:mini} and \Cref{prob:stabpkpk}, or their corresponding standard counterparts where the right hand side does not include the Nédélec interpolation of the force $f$. To highlight the difference we will use the symbols $u_h^\pr$ and $p_h^\pr$ for the pressure-robust methods, and $u^{\std}_h$ and $p^{\std}_h$ for the standard methods.

\subsection{The MINI element and the stabilized $P^2P^2$ element}
\label{ssec:numex_mini_and_P2P2}

\begin{table}[h]
	\centering
	\renewcommand{\arraystretch}{1.1}
	\begin{tabular}{c|c|c|c|c|c|c}
		\midrule\midrule
		\multicolumn{7}{c}{ MINI element -- standard method -- $\nu = 1$} \\
		$|\mathcal{T}_h|$ & $\|u-u^\std_h\|_{L^2}$ & eoc & $\|\nabla(u-u_h^\std)\|_{L^2}$ & eoc & $\|p-p^\std_h\|_{L^2}$ & eoc \\
		\hline
		\nummesh{230}   & \num{3.40e-04} & --   & \num{1.18e-02} & --   & \num{6.68e-03} & --   \\
		\nummesh{920}   & \num{1.38e-04} & \numeoc{1.31} & \num{7.33e-03} & \numeoc{0.69} & \num{5.68e-03} & \numeoc{0.23} \\
		\nummesh{3680}  & \num{3.23e-05} & \numeoc{2.09} & \num{3.51e-03} & \numeoc{1.06} & \num{2.12e-03} & \numeoc{1.42} \\
		\nummesh{14720} & \num{7.73e-06} & \numeoc{2.06} & \num{1.71e-03} & \numeoc{1.04} & \num{7.60e-04} & \numeoc{1.48}\\
		\midrule
		\multicolumn{7}{c}{ MINI element -- standard method -- $\nu = 10^{-6}$} \\
		$|\mathcal{T}_h|$ & $\|u-u^\std_h\|_{L^2}$ & eoc & $\|\nabla(u-u_h^\std)\|_{L^2}$ & eoc & $\|p-p^\std_h\|_{L^2}$ & eoc \\
		\hline
		\nummesh{230}   & \num{1.78e+01} & --   & \num{1.32e+03} & --   & \num{3.05e-03} & --   \\
		\nummesh{920}   & \num{3.75e+00} & \numeoc{2.25} & \num{4.68e+02} & \numeoc{1.50} & \num{1.15e-03} & \numeoc{1.41} \\
		\nummesh{3680}  & \num{4.20e-01} & \numeoc{3.16} & \num{1.11e+02} & \numeoc{2.07} & \num{2.84e-04} & \numeoc{2.01} \\
		\nummesh{14720} & \num{4.66e-02} & \numeoc{3.17} & \num{2.62e+01} & \numeoc{2.08} & \num{7.03e-05} & \numeoc{2.02} \\
		\midrule\midrule
	\end{tabular}
	\caption{
	Error convergence for the solutions $u_h^\std$ and $p_h^\std$ of the standard method using the MINI element and viscosities $\nu = 1$ and $\nu = 10^{-6}$.}
	\label{tab:mini_std}
\end{table}
\begin{table}[h]
	\centering
	\renewcommand{\arraystretch}{1.1}
	\begin{tabular}{c|c|c|c|c|c|c}
		\midrule\midrule
		\multicolumn{7}{c}{ MINI element -- pressure-robust method -- $\nu = 1$} \\
		$|\mathcal{T}_h|$ & $\|u-u^\pr_h\|_{L^2}$ & eoc & $\|\nabla(u-u^\pr_h)\|_{L^2}$ & eoc & $\|p-p^\pr_h\|_{L^2}$ & eoc \\
		\hline
		\nummesh{230}   & \num{4.17e-04} & --   & \num{1.18e-02} & --   & \num{8.61e-03} & --   \\
		\nummesh{920}   & \num{1.60e-04} & \numeoc{1.38} & \num{7.32e-03} & \numeoc{0.68} & \num{5.86e-03} & \numeoc{0.55} \\
		\nummesh{3680}  & \num{3.81e-05} & \numeoc{2.07} & \num{3.51e-03} & \numeoc{1.06} & \num{2.15e-03} & \numeoc{1.44} \\
		\nummesh{14720} & \num{9.17e-06} & \numeoc{2.05} & \num{1.71e-03} & \numeoc{1.04} & \num{7.66e-04} & \numeoc{1.49} \\
		\midrule
		\multicolumn{7}{c}{ MINI element -- pressure-robust method -- $\nu = 10^{-6}$} \\
		$|\mathcal{T}_h|$ & $\|u-u^\pr_h\|_{L^2}$ & eoc & $\|\nabla(u-u^\pr_h)\|_{L^2}$ & eoc & $\|p-p^\pr_h\|_{L^2}$ & eoc \\
		\hline
		\nummesh{230}   & \num{4.17e-04} & --   & \num{1.18e-02} & --   & \num{6.30e-03} & --   \\
		\nummesh{920}   & \num{1.60e-04} & \numeoc{1.38} & \num{7.32e-03} & \numeoc{0.68} & \num{1.93e-03} & \numeoc{1.71} \\
		\nummesh{3680}  & \num{3.81e-05} & \numeoc{2.07} & \num{3.51e-03} & \numeoc{1.06} & \num{4.84e-04} & \numeoc{1.99} \\
		\nummesh{14720} & \num{9.17e-06} & \numeoc{2.05} & \num{1.71e-03} & \numeoc{1.04} & \num{1.21e-04} & \numeoc{2.00} \\
		\midrule\midrule
	\end{tabular}
	\caption{Error convergence for the solutions $u^\pr_h$ and $p_h^\pr$ of the pressure-robust method (\Cref{prob:mini}) using the MINI element and viscosities $\nu = 1$ and $\nu = 10^{-6}$.}
	\label{tab:mini_robust}
\end{table}

\Cref{tab:mini_std,tab:mini_robust} report the convergence of the two (standard and pressure-robust) methods on the same sequence of meshes for the MINI finite element method.

We first start with the pressure-robust method (\Cref{tab:mini_robust}) and note that the errors are independent of the viscosity $\nu$ and converge at optimal order. The $H^1$- and $L^2$-velocity error converges at linear and quadratic order, respectively, while the pressure error converges at a slightly more than linear order for $\nu = 1$ and even at quadratic order for $\nu = 10^{-6}$. The latter can be explained by the estimates from \Cref{thm:mini_pressure} as discussed in \Cref{rem:mini_superconvergence}: the linear, velocity dependent term is scaled by the viscosity and thus -- depending on the order of magnitude of the $|u|_{H^2} + |\Delta u|_{H^1}$ - can be dominated by the $h^2|p|_{H^2}$ term. This is a pre-asymptotic super-convergence effect, which is particularly more pronounced for small viscosities.

Similar observations can be made for the standard method (\Cref{tab:mini_std}) at $\nu = 1$: all errors converge at optimal order, i.e. we see a linear and quadratic order of the $H^1$- and $L^2$-velocity error, respectively, and an increased order of the pressure error. The latter can be explained by a similar pre-asymptotic super-convergence behavior as in the pressure-robust case. Most crucially, the velocity errors of the standard methods are \textit{not} independent of the viscosity: for $\nu = 10^{-6}$, the velocity errors are polluted by the pressure error scaled by the inverse of the viscosity and thus get much bigger -- in fact, this pollution even raises the observed eoc above the true rates $1$ and $2$, since by \eqref{eq:classical_error} the $\nu^{-1}h^2$-scaled term dominates the true $h$-term over this range of $h$.

The picture for the stabilized $P^2P^2$ element (\Cref{tab:p2p2_std,tab:p2p2_robust})
parallels the MINI case: the $H^1$- and $L^2$-velocity errors of the pressure-robust method
in \Cref{tab:p2p2_robust} are again independent of the choice of $\nu$, while the standard method of \Cref{tab:p2p2_std} is off by roughly three to four orders of magnitude at every mesh level despite eoc values that provide an optimal (or even higher pre-asymptotic) order of convergence.

\begin{table}[h]
	\centering
	\renewcommand{\arraystretch}{1.1}
	\begin{tabular}{c|c|c|c|c|c|c}
		\midrule\midrule
		\multicolumn{7}{c}{ Stabilized $P^2P^2$ element -- standard method -- $\nu = 1$} \\
		$|\mathcal{T}_h|$ & $\|u-u^\std_h\|_{L^2}$ & eoc & $\|\nabla(u-u_h^\std)\|_{L^2}$ & eoc & $\|p-p^\std_h\|_{L^2}$ & eoc \\
		\hline
		\nummesh{230}   & \num{2.26e-05} & --   & \num{1.43e-03} & --   & \num{5.01e-03} & --   \\
		\nummesh{920}   & \num{4.26e-06} & \numeoc{2.41} & \num{4.66e-04} & \numeoc{1.62} & \num{6.28e-04} & \numeoc{2.99} \\
		\nummesh{3680}  & \num{5.51e-07} & \numeoc{2.95} & \num{1.17e-04} & \numeoc{1.99} & \num{1.47e-04} & \numeoc{2.10} \\
		\nummesh{14720} & \num{7.02e-08} & \numeoc{2.97} & \num{2.93e-05} & \numeoc{2.00} & \num{3.53e-05} & \numeoc{2.06} \\
		\midrule
		\multicolumn{7}{c}{ Stabilized $P^2P^2$ element -- standard method -- $\nu = 10^{-6}$} \\
		$|\mathcal{T}_h|$ & $\|u-u^\std_h\|_{L^2}$ & eoc & $\|\nabla(u-u_h^\std)\|_{L^2}$ & eoc & $\|p-p^\std_h\|_{L^2}$ & eoc \\
		\hline
		\nummesh{230}   & \num{5.80e-01} & --   & \num{4.86e+01} & --   & \num{1.99e-04} & --   \\
		\nummesh{920}   & \num{8.74e-02} & \numeoc{2.73} & \num{1.35e+01} & \numeoc{1.85} & \num{3.83e-05} & \numeoc{2.37} \\
		\nummesh{3680}  & \num{6.90e-03} & \numeoc{3.66} & \num{2.09e+00} & \numeoc{2.69} & \num{4.73e-06} & \numeoc{3.02} \\
		\nummesh{14720} & \num{5.00e-04} & \numeoc{3.78} & \num{2.99e-01} & \numeoc{2.81} & \num{5.63e-07} & \numeoc{3.07} \\
		\midrule\midrule
	\end{tabular}
	\caption{Error convergence for the solutions $u_h^\std$ and $p_h^\std$ of the standard
	method using the stabilized $P^2P^2$ element ($\alpha=10^{-3}$) and viscosities $\nu = 1$
	and $\nu = 10^{-6}$.}
	\label{tab:p2p2_std}
\end{table}
\begin{table}[h]
	\centering
	\renewcommand{\arraystretch}{1.1}
	\begin{tabular}{c|c|c|c|c|c|c}
		\midrule\midrule
		\multicolumn{7}{c}{ Stabilized $P^2P^2$ element -- pressure-robust method -- $\nu = 1$} \\
		$|\mathcal{T}_h|$ & $\|u-u^\pr_h\|_{L^2}$ & eoc & $\|\nabla(u-u^\pr_h)\|_{L^2}$ & eoc & $\|p-p^\pr_h\|_{L^2}$ & eoc \\
		\hline
		\nummesh{230}   & \num{2.28e-05} & --   & \num{1.43e-03} & --   & \num{5.01e-03} & --   \\
		\nummesh{920}   & \num{4.27e-06} & \numeoc{2.41} & \num{4.66e-04} & \numeoc{1.62} & \num{6.29e-04} & \numeoc{2.99} \\
		\nummesh{3680}  & \num{5.51e-07} & \numeoc{2.96} & \num{1.17e-04} & \numeoc{1.99} & \num{1.47e-04} & \numeoc{2.10} \\
		\nummesh{14720} & \num{7.02e-08} & \numeoc{2.97} & \num{2.93e-05} & \numeoc{2.00} & \num{3.53e-05} & \numeoc{2.06} \\
		\midrule
		\multicolumn{7}{c}{ Stabilized $P^2P^2$ element -- pressure-robust method -- $\nu = 10^{-6}$} \\
		$|\mathcal{T}_h|$ & $\|u-u^\pr_h\|_{L^2}$ & eoc & $\|\nabla(u-u^\pr_h)\|_{L^2}$ & eoc & $\|p-p^\pr_h\|_{L^2}$ & eoc \\
		\hline
		\nummesh{230}   & \num{2.28e-05} & --   & \num{1.43e-03} & --   & \num{1.50e-04} & --   \\
		\nummesh{920}   & \num{4.27e-06} & \numeoc{2.41} & \num{4.66e-04} & \numeoc{1.62} & \num{2.94e-05} & \numeoc{2.35} \\
		\nummesh{3680}  & \num{5.51e-07} & \numeoc{2.96} & \num{1.17e-04} & \numeoc{1.99} & \num{3.69e-06} & \numeoc{3.00} \\
		\nummesh{14720} & \num{7.02e-08} & \numeoc{2.97} & \num{2.93e-05} & \numeoc{2.00} & \num{4.62e-07} & \numeoc{3.00} \\
		\midrule\midrule
	\end{tabular}
	\caption{Error convergence for the solutions $u_h^\pr$ and $p_h^\pr$ of the
	pressure-robust method (\Cref{prob:stabpkpk}) using the stabilized $P^2P^2$ element
	($\alpha=10^{-3}$) and viscosities $\nu = 1$ and $\nu = 10^{-6}$.}
	\label{tab:p2p2_robust}
\end{table}

\subsection{The Taylor--Hood element}
\label{ssec:numex_th}
\begin{table}[h]
	\centering
	\renewcommand{\arraystretch}{1.1}
	\begin{tabular}{c|c|c|c|c|c|c}
		\midrule\midrule
		\multicolumn{7}{c}{ Taylor--Hood element ($P^2P^1$) -- standard method -- $\nu = 1$} \\
		$|\mathcal{T}_h|$ & $\|u-u^\std_h\|_{L^2}$ & eoc & $\|\nabla(u-u_h^\std)\|_{L^2}$ & eoc & $\|p-p^\std_h\|_{L^2}$ & eoc \\
		\hline
		\nummesh{230}   & \num{1.90e-05} & --   & \num{1.36e-03} & --   & \num{3.03e-03} & --   \\
		\nummesh{920}   & \num{6.71e-06} & \numeoc{1.50} & \num{7.13e-04} & \numeoc{0.93} & \num{1.08e-03} & \numeoc{1.49} \\
		\nummesh{3680}  & \num{8.28e-07} & \numeoc{3.02} & \num{1.79e-04} & \numeoc{1.99} & \num{2.72e-04} & \numeoc{1.99} \\
		\nummesh{14720} & \num{9.79e-08} & \numeoc{3.08} & \num{4.39e-05} & \numeoc{2.03} & \num{6.77e-05} & \numeoc{2.00} \\
		\midrule
		\multicolumn{7}{c}{ Taylor--Hood element ($P^2P^1$) -- standard method -- $\nu = 10^{-6}$} \\
		$|\mathcal{T}_h|$ & $\|u-u^\std_h\|_{L^2}$ & eoc & $\|\nabla(u-u_h^\std)\|_{L^2}$ & eoc & $\|p-p^\std_h\|_{L^2}$ & eoc \\
		\hline
		\nummesh{230}   & \num{1.13e+01} & --   & \num{7.32e+02} & --   & \num{3.02e-03} & --   \\
		\nummesh{920}   & \num{5.86e+00} & \numeoc{0.95} & \num{5.87e+02} & \numeoc{0.32} & \num{1.08e-03} & \numeoc{1.49} \\
		\nummesh{3680}  & \num{7.15e-01} & \numeoc{3.03} & \num{1.47e+02} & \numeoc{2.00} & \num{2.70e-04} & \numeoc{2.00} \\
		\nummesh{14720} & \num{8.27e-02} & \numeoc{3.11} & \num{3.56e+01} & \numeoc{2.04} & \num{6.73e-05} & \numeoc{2.01} \\
		\midrule\midrule
	\end{tabular}
	\caption{Error convergence for the solutions $u_h^\std$ and $p_h^\std$ of the standard
	method \eqref{eq:disc_stokes} using the Taylor--Hood element ($P^2P^1$) and viscosities
	$\nu = 1$ and $\nu = 10^{-6}$.}
	\label{tab:th_std}
\end{table}
\begin{table}[h]
	\centering
	\renewcommand{\arraystretch}{1.1}
	\begin{tabular}{c|c|c|c|c|c|c}
		\midrule\midrule
		\multicolumn{7}{c}{ Taylor--Hood element ($P^2P^1$) -- pressure-robust method -- $\nu = 1$} \\
		$|\mathcal{T}_h|$ & $\|u-u^\pr_h\|_{L^2}$ & eoc & $\|\nabla(u-u^\pr_h)\|_{L^2}$ & eoc & $\|p-p^\pr_h\|_{L^2}$ & eoc \\
		\hline
		\nummesh{230}   & \num{9.88e-05} & --   & \num{1.37e-03} & --   & \num{6.21e-03} & --   \\
		\nummesh{920}   & \num{2.76e-05} & \numeoc{1.84} & \num{5.13e-04} & \numeoc{1.41} & \num{1.91e-03} & \numeoc{1.70} \\
		\nummesh{3680}  & \num{6.82e-06} & \numeoc{2.02} & \num{1.29e-04} & \numeoc{2.00} & \num{4.79e-04} & \numeoc{1.99} \\
		\nummesh{14720} & \num{1.70e-06} & \numeoc{2.00} & \num{3.21e-05} & \numeoc{2.00} & \num{1.20e-04} & \numeoc{2.00} \\
		\midrule
		\multicolumn{7}{c}{ Taylor--Hood element ($P^2P^1$) -- pressure-robust method -- $\nu = 10^{-6}$} \\
		$|\mathcal{T}_h|$ & $\|u-u^\pr_h\|_{L^2}$ & eoc & $\|\nabla(u-u^\pr_h)\|_{L^2}$ & eoc & $\|p-p^\pr_h\|_{L^2}$ & eoc \\
		\hline
		\nummesh{230}   & \num{9.88e-05} & --   & \num{1.37e-03} & --   & \num{6.30e-03} & --   \\
		\nummesh{920}   & \num{2.76e-05} & \numeoc{1.84} & \num{5.13e-04} & \numeoc{1.41} & \num{1.93e-03} & \numeoc{1.71} \\
		\nummesh{3680}  & \num{6.82e-06} & \numeoc{2.02} & \num{1.29e-04} & \numeoc{2.00} & \num{4.84e-04} & \numeoc{1.99} \\
		\nummesh{14720} & \num{1.70e-06} & \numeoc{2.00} & \num{3.21e-05} & \numeoc{2.00} & \num{1.21e-04} & \numeoc{2.00} \\
		\midrule\midrule
	\end{tabular}
	\caption{Error convergence for the solutions $u_h^\pr$ and $p_h^\pr$ of the
	pressure-robust method (the analogue of \Cref{prob:mini} using $\NedI{0} f$) using the
	Taylor--Hood element ($P^2P^1$) and viscosities $\nu = 1$ and $\nu = 10^{-6}$.}
	\label{tab:th_robust}
\end{table}

As discussed in \Cref{sec:taylorhood}, the Taylor--Hood element is the one where we expect
a reduced order in the $L^2$-norm error of the velocity for the pressure-robust method. The numerical results in \Cref{tab:th_std,tab:th_robust} confirm this expectation for $k = 1$, i.e. a quadratic polynomial order for the velocity space and a linear order for the pressure space. As for the other two elements, the $H^1$-velocity and pressure errors converge at optimal (quadratic) order for both the standard and the pressure-robust method, and the velocity errors of the pressure-robust method are independent of $\nu$. Unfortunately, the $L^2$-velocity error of the pressure-robust method converges only quadratically, while the standard method converges cubically. However, we want to emphasize that, even though the $L^2$-velocity error of the pressure-robust method converges at a lower order, it is still much smaller than the $L^2$-velocity error of the standard method for small viscosities.

\subsection{The three dimensional case}
\label{sec:numerics_3d}

As discussed in the introduction, the analysis can be extended also to the three dimensional case. For the lowest order case, e.g. the $P^1P^1$ element, the interpolation operator is again the standard Nédélec interpolation operator $\NedI{0}$ and the analysis is exactly the same. For the higher order case, the moments used to define the interpolation operator $\NedI{k-1}$ slightly differs, thus the analysis of \Cref{prop:consistency_high} is slightly different and even more technical. However, the main idea is the same. To demonstrate the applicability of the proposed method in three dimensions, we present a numerical example in the following.
For this we choose the unit cube $\Omega\coloneqq(0,1)^3$ and let the exact (smooth) velocity and pressure be given by
$u(x,y,z)\coloneqq\curl\big(x^2(1-x)^2y^2(1-y)^2z^2(1-z)^2\big)\in [H^1_0(\Omega)]^3$ and $p(x,y,z)\coloneqq x^5+y^5+z^5-1/2\in L^2_0(\Omega) $.
We consider an initial structured mesh of $196$ tetrahedra which is subsequently uniformly refined.

\begin{table}[h]
	\centering
	\renewcommand{\arraystretch}{1.1}
	\begin{tabular}{c|c|c|c|c|c|c}
		\midrule\midrule
		\multicolumn{7}{c}{ Stabilized $P^1P^1$ element -- standard method -- $\nu = 1$} \\
		$|\mathcal{T}_h|$ & $\|u-u^\std_h\|_{L^2}$ & eoc & $\|\nabla(u-u_h^\std)\|_{L^2}$ & eoc & $\|p-p^\std_h\|_{L^2}$ & eoc \\
		\hline
		\nummesh{196}    & \num{4.61e-04} & --   & \num{3.96e-03} & --   & \num{4.92e-02} & --   \\
		\nummesh{1568}   & \num{2.17e-04} & \numeoc{1.09} & \num{3.11e-03} & \numeoc{0.35} & \num{1.92e-02} & \numeoc{1.36} \\
		\nummesh{12544}  & \num{7.17e-05} & \numeoc{1.60} & \num{1.74e-03} & \numeoc{0.84} & \num{6.64e-03} & \numeoc{1.53} \\
		\nummesh{100352} & \num{1.81e-05} & \numeoc{1.99} & \num{7.78e-04} & \numeoc{1.16} & \num{2.07e-03} & \numeoc{1.68} \\
		\midrule
		\multicolumn{7}{c}{ Stabilized $P^1P^1$ element -- standard method -- $\nu = 10^{-6}$} \\
		$|\mathcal{T}_h|$ & $\|u-u^\std_h\|_{L^2}$ & eoc & $\|\nabla(u-u_h^\std)\|_{L^2}$ & eoc & $\|p-p^\std_h\|_{L^2}$ & eoc \\
		\hline
		\nummesh{196}    & \num{7.87e+01} & --   & \num{9.68e+02} & --   & \num{4.67e-02} & --   \\
		\nummesh{1568}   & \num{6.43e+01} & \numeoc{0.29} & \num{1.77e+03} & \numeoc{-0.87} & \num{1.79e-02} & \numeoc{1.38} \\
		\nummesh{12544}  & \num{1.75e+01} & \numeoc{1.88} & \num{9.39e+02} & \numeoc{0.92} & \num{5.77e-03} & \numeoc{1.64} \\
		\nummesh{100352} & \num{2.18e+00} & \numeoc{3.01} & \num{2.38e+02} & \numeoc{1.98} & \num{1.49e-03} & \numeoc{1.95} \\
		\midrule\midrule
	\end{tabular}
	\caption{Error convergence for the solutions $u_h^\std$ and $p_h^\std$ of the standard
	method using the stabilized $P^1P^1$ element and viscosities $\nu = 1$ and $\nu = 10^{-6}$ in three dimensions.}
	\label{tab:p1p1_3d_std}
\end{table}
\begin{table}[h]
	\centering
	\renewcommand{\arraystretch}{1.1}
	\begin{tabular}{c|c|c|c|c|c|c}
		\midrule\midrule
		\multicolumn{7}{c}{ Stabilized $P^1P^1$ element -- pressure-robust method -- $\nu = 1$} \\
		$|\mathcal{T}_h|$ & $\|u-u^\pr_h\|_{L^2}$ & eoc & $\|\nabla(u-u^\pr_h)\|_{L^2}$ & eoc & $\|p-p^\pr_h\|_{L^2}$ & eoc \\
		\hline
		\nummesh{196}    & \num{4.89e-04} & --   & \num{4.04e-03} & --   & \num{7.26e-02} & --   \\
		\nummesh{1568}   & \num{2.36e-04} & \numeoc{1.05} & \num{2.64e-03} & \numeoc{0.62} & \num{2.57e-02} & \numeoc{1.50} \\
		\nummesh{12544}  & \num{7.82e-05} & \numeoc{1.59} & \num{1.47e-03} & \numeoc{0.84} & \num{8.30e-03} & \numeoc{1.63} \\
		\nummesh{100352} & \num{2.06e-05} & \numeoc{1.93} & \num{7.41e-04} & \numeoc{0.99} & \num{2.46e-03} & \numeoc{1.75} \\
		\midrule
		\multicolumn{7}{c}{ Stabilized $P^1P^1$ element -- pressure-robust method -- $\nu = 10^{-6}$} \\
		$|\mathcal{T}_h|$ & $\|u-u^\pr_h\|_{L^2}$ & eoc & $\|\nabla(u-u^\pr_h)\|_{L^2}$ & eoc & $\|p-p^\pr_h\|_{L^2}$ & eoc \\
		\hline
		\nummesh{196}    & \num{4.89e-04} & --   & \num{4.04e-03} & --   & \num{7.32e-02} & --   \\
		\nummesh{1568}   & \num{2.36e-04} & \numeoc{1.05} & \num{2.64e-03} & \numeoc{0.62} & \num{2.53e-02} & \numeoc{1.54} \\
		\nummesh{12544}  & \num{7.82e-05} & \numeoc{1.59} & \num{1.47e-03} & \numeoc{0.84} & \num{7.77e-03} & \numeoc{1.70} \\
		\nummesh{100352} & \num{2.06e-05} & \numeoc{1.93} & \num{7.41e-04} & \numeoc{0.99} & \num{2.04e-03} & \numeoc{1.93} \\
		\midrule\midrule
	\end{tabular}
	\caption{Error convergence for the solutions $u_h^\pr$ and $p_h^\pr$ of the
	pressure-robust method (\Cref{prob:stabpkpk}) using the stabilized $P^1P^1$ element and viscosities $\nu = 1$ and $\nu = 10^{-6}$ in three dimensions.}
	\label{tab:p1p1_3d_robust}
\end{table}

\Cref{tab:p1p1_3d_std,tab:p1p1_3d_robust} report the convergence of the standard and pressure-robust methods, respectively, on the same sequence of tetrahedral meshes for the (lowest order) stabilized $P^1P^1$ element (again with $\alpha = 10^{-3}$). As in the two dimensional case, the pressure-robust method (\Cref{tab:p1p1_3d_robust}) yields velocity errors that are independent of the viscosity $\nu$, while the standard method (\Cref{tab:p1p1_3d_std}) suffers from severe pressure-induced pollution of the velocity error for $\nu = 10^{-6}$; on the coarser meshes this pollution is so strong that the $H^1$-velocity error even briefly increases under refinement before the asymptotic rate sets in. Apart from this pre-asymptotic effect, the observed rates for both methods agree with the two dimensional case, confirming that the pressure-robustness mechanism and the underlying analysis carry over to three dimensions.

\bibliographystyle{amsplain}
\bibliography{references}

\end{document}